\documentclass[a4paper,12pt]{amsart}

\usepackage{anysize}

\newcommand{\phof}{P_{\mathrm{Hof}}}
\newcommand{\pmar}{P_{\mathrm{Mar}}}

\usepackage{pdfsync}
\usepackage[active]{srcltx}
\usepackage[english]{babel}
\usepackage{amscd}
\usepackage{amssymb}
\usepackage{amsthm}
\usepackage{amsmath}
\usepackage[latin2]{inputenc}
\usepackage{t1enc}
\usepackage{graphicx}
\usepackage{comment}
\usepackage{enumerate}
\usepackage{hyperref}
\usepackage{psfrag}
\usepackage{epsfig}
\usepackage{dsfont}
\usepackage{enumerate}
\usepackage{bbm}
\usepackage{caption}
\usepackage{subcaption}
\usepackage[usenames,dvipsnames,svgnames]{xcolor}

\usepackage{comment}

\usepackage{wrapfig}
\usepackage{float}

\newtheorem{theorem}{Theorem}[section]
\theoremstyle{plain}

\newtheorem*{claim}{Claim}

\newtheorem{corollary}[theorem]{Corollary}

\newtheorem{definition}[theorem]{Definition}

\newtheorem*{quest}{Question}

\newtheorem{lemma}[theorem]{Lemma}

\newtheorem{proposition}[theorem]{Proposition}
\newtheorem{remark}[theorem]{Remark}

\definecolor{trp}{rgb}{1,1,1}

\definecolor{red}{rgb}{1,0,.2}

\definecolor{blue}{rgb}{0,0,1}

\definecolor{rgrey}{rgb}{.8,0.4,.4}  
\definecolor{grey}{rgb}{.13,.13,.13}  

\definecolor{green}{rgb}{0.0,0.4,0.2}

\newcommand{\DD}{\mathcal{D}}
\newcommand{\MM}{\mathcal{M}}
\newcommand{\II}{\mathcal{I}}
\newcommand{\JJ}{\mathcal{J}}
\newcommand{\vv}{\underline{v}}
\newcommand{\wv}{\underline{w}}
\newcommand{\pv}{\underline{p}}
\newcommand{\RP}{{\mathbb{RP}^1}}

\usepackage{fancyhdr}
\numberwithin{equation}{section}

\newcommand{\R}{\mathbb{R}}
\newcommand{\N}{\mathbb{N}}
\newcommand{\xv}{\underline{x}}
\newcommand{\zv}{\underline{z}}
\newcommand{\yv}{\underline{y}}
\newcommand{\tv}{\mathbf{t}}
\newcommand{\proj}{\mathrm{proj}}
\newcommand{\ii}{\mathbf{i}}
\newcommand{\jj}{\mathbf{j}}
\newcommand{\iv}{{\overline{\imath}}}
\newcommand{\jv}{{\overline{\jmath}}}
\newcommand{\ev}{\underline{e}}

\newcommand{\xvv}{\mathbf{x}}
\newcommand{\tvv}{\mathbf{t}}

\begin{document}

\pagestyle{myheadings}

\title[Dimension of the repeller for a piecewise expanding affine map]{Dimension of the repeller for a piecewise expanding affine map}


\author{Bal\'azs B\'ar\'any}
\address[Bal\'azs B\'ar\'any]{Budapest University of Technology and Economics, Department of Stochastics, MTA-BME Stochastics Research Group, P.O.Box 91, 1521 Budapest, Hungary}
\email{balubsheep@gmail.com}

\author{Micha\l\ Rams}
\address[Micha\l\ Rams]{Institute of Mathematics, Polish Academy of Sciences, ul. \'Sniadeckich 8, 00-656 Warszawa, Poland}
\email{rams@impan.pl}

\author{K\'aroly Simon}
\address[K\'aroly Simon]{Budapest University of Technology and Economics, Department of Stochastics, Institute of Mathematics, 1521 Budapest, P.O.Box 91, Hungary} \email{simonk@math.bme.hu}

\thanks{The research of B\'ar\'any and Simon was supported by the grant OTKA K123782 and by the National Research, Development and Innovation Fund (TUDFO/51757/2019-ITM, Thematic Excellence Program). B\'ar\'any acknowledges support also from NKFI PD123970 and the J\'anos Bolyai Research Scholarship of the Hungarian Academy of Sciences. Micha\l\ Rams was supported by National Science Centre grant 2014/13/B/ST1/01033 (Poland). This work was also supported by the  grant  346300 for IMPAN from the Simons Foundation and the matching 2015-2019 Polish MNiSW fund.}

\begin{abstract}
In this paper, we study the dimension theory of a class of piecewise affine systems in euclidean spaces suggested by Michael Barnsley, with some applications to the fractal image compression. It is a more general version of the class considered in the work of Keane, Simon and Solomyak \cite{KSS03} and can be considered as the continuation of the works \cite{barany2016dimension, BRS2} by the authors. We also present some applications of our results for generalized Takagi functions and fractal interpolation functions.
\end{abstract}

\date{\today}

\maketitle
\section{Introduction and Statements}\label{sec:intro}

The self-affine fractal graphs, an example of which, the Takagi function, is depicted in Figure~\ref{fig:taka}, are defined as function graphs which are equal to the finite union of their own affine copies. Motivated by a question of Michael Barnsley, in this paper, we consider a more general family of dynamically defined fractal graphs. A conspicuous example can be the generalized Takagi-like function $x\mapsto T(x)$, see Figure~\ref{d80},for the precise definition see \cite[Section~4]{BRSpartII}. The graph of the function $T$ is the union of affine images of certain subsets of the graph of $T$. Namely, the left-hand side of the graph is an affine image of the part of the graph above interval $J_1$ and similarly, the right-hand side is the affine image of the part above $J_2$. If the graph was self-affine, the left- and right-hand side of the graph would be the union of the images of the whole graph. This cause significant difficulties in the understand of the underlying dynamics. The novelty of the paper is that we introduce techniques with which these more complicated objects can be treated. A friendly presentation to the ideas, introduced in this paper, without any technical details but with more examples, can be found in the survey papers \cite{BRSpartI} and \cite{BRSpartII}.

\begin{figure}
	\centering
    \includegraphics[width=12cm]{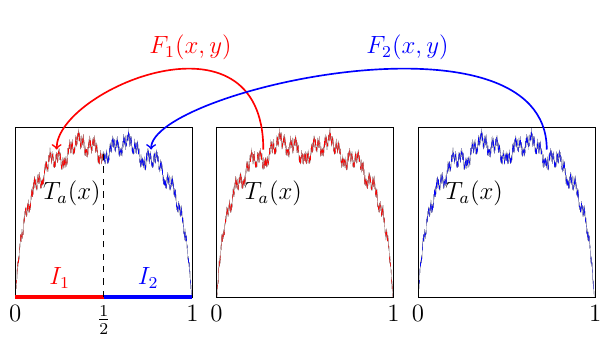}
    \caption{Self-affine Takagi function.}\label{fig:taka}
	\includegraphics[width=12cm]{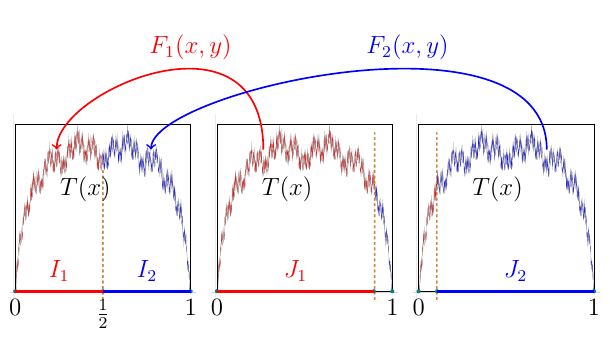}
	\caption{The generalized Takagi-like function with non-Markovian structure.}\label{d80}
\end{figure}

Our research was motivated by a fractal image compression problem. A monochrome picture can be imagined as a function $G\colon[0,1]^2\mapsto[0,1]$, where $G(x,y)$ represents the tone at the point $(x,y)$. Similarly, a colored picture can be imagined as $G\colon[0,1]^2\mapsto[0,1]^3$ using RGB colors as coordinates. One way to encode the picture is the fractal image compression method, which concept was first introduced by Barnsley, see \cite{Barnsley86, Barnsley96} and Barnsley and Hurd \cite{BarnsleyHurd} and Barnsley and Elton \cite{BarnsleyElton}. Later, the theory developed widely, see for example Fisher \cite{Fisher}, Keane, Simon and Solomyak \cite{KSS03}, Chung and Hsu~\cite{ChungHsu}, Jorgensen and Song \cite{JorgensenSong}.

The idea behind the fractal image compression is that a natural image, such as a face, landscape etc., contains a kind of self-similarity. Part of the image is similar to another part of the image. It does not contain rigorous self-similarity, the image does not appear to contain affine transformations of itself. So, rather than having the image be formed of copies of its whole self (under appropriate affine transformation), the image will be formed of copies of properly transformed parts of itself. This is not to mean that every natural image is such a fractal -- surely they are not. However, we can approximate the original image with images that can be encoded with as the attractor/repeller of such piecewise affine dynamical system, and those images can then be decoded by the coefficients of the maps used to construct them. If we started with an image that is locally self-similar (or at least some parts of it are), this can be a very short code, see Lu \cite{LU_frac}.

Jacquin \cite{Jacquin} proposed to partition the image into non-overlapping blocks called range blocks and find an appropriate partitioned iterated function system for each range. More precisely, the fractal image compression decomposes $[0,1]^2$ into axes parallel rectangles $\{I_i\}$ and finds uniformly expanding linear functions $f_i\colon I_i\mapsto[0,1]^2$ and $F_i:I_i\times\R\mapsto [0,1]^2\times\R$ such that $F_i(\{(x,y,G(x,y)):(x,y)\in I_i\})$ is close to the set $\{(x,y,G(x,y)):(x,y)\in f_i(I_i)\}$. Then repeller $\Lambda$ of $F_i$ is close to the set $\mathrm{graph}(G)$, which had to be approximated.

The fractal image compression requires each range block to be compared with all possible domain blocks within the image. Several papers proposed to use the local fractal dimension of image partitions to check for similarity, see Conci and Aquino \cite{AuraAquino} and Revathy and Jayamohan \cite{RevathyJayamohan}. This motivates our paper, which question was also raised by Michael Barnsley \cite{Barnsleyspeak}.
\begin{quest}
What is the Hausdorff dimension of the set $\Lambda$?
\end{quest}

Some of the fractal image compression algorithms use the "fractal dimension" to cluster the original picture and to find the best fitting maps $F_i$, see \cite{LU_frac} and \cite{Posada}. Usually, the box counting dimension is used in the applied papers on the fractal image compression, however, from the theoretical point of view, the Hausdorff dimension plays important role. We answer the question for the case, when $f_i$ are linear, conformal maps such that the map $f(x,y)=f_i(x,y)$ if $(x,y)\in I_i$ is Markov. Moreover, we give a complete answer on the line, namely, when $\{I_i\}$ is an interval decomposition of the unit interval and $f_i\colon I_i\mapsto[0,1]$ arbitrary, expanding linear functions. The previously known results on such objects were almost every type, namely, there is a set of parameters with full measure so that the dimension could be calculated on this set. We give here every type results, namely, under some mild conditions, the dimension can be calculated. This gives a feedback for the algorithm given for example in \cite{Posada}, how the chosen parameters approximate the original picture in sense of dimension. On the other hand, if we want to find a certain picture in a large collection of pictures coded with fractal image compression then it helps if we know the value of the dimension of the corresponding repeller. Namely, one can restrict the collection of the pictures (with fast algoritms) to a subcollection, which dimension is close to the searched picture. To do so, we need to apply some very recent tools from fractal geometry and one dimensional dynamics. For instance, considering the generalized Takagi-like function $x\mapsto T(x)$, see Figure~\ref{d80}, the graph is not self-affine, the corresponding dynamics does not admit a Markov partition, as opposed to the traditional generalized Takagi-function, see \cite{BRSpartII}. Hence, we need to admit the techniques introduced by Hofbauer to tackle the problems caused by the non-Markovian structure together with the recent techniques in the study of self-affine sets.

Let us now describe the setup of our paper in more details. Let $d\geq1$ and let $\{I_i\}_{i=1}^M$ be a partition of the unit cube $I:=[0,1]^d$. Assume that all the elements are regular sets in the sense that $\overline{I_i^o}=\overline{I_i}$ for every $i=1,\ldots,M$, where $A^o$ denotes the interior and $\overline{A}$ denotes the closure of the set $A$. For simplicity, we assume that all the sets of the partition are simply connected.

 Moreover, for every $i=1,\dots,M$ let $f_i$ be a uniformly expanding similitude of the form
 $$
 f_i(\xv)=\gamma_iU_i\xv+\vv_i,
 $$
 where $\gamma_i>1$, $U_i\in O(d)$ and $\vv_i\in\R^d$, such that $f_i\colon I_i\mapsto I$ and $f_i$ can be extended to $\overline{I_i}$, and we consider the uniformly, piecewise expanding dynamical system $f\colon I\mapsto I$, where
\begin{equation}\label{eq:base}
f(x)=f_i(x)\text{ if }x\in I_i.
\end{equation}
We say that $f$ is \textit{Markov} if $f(\overline{I_i})$ is equal to a finite union of elements in $\{\overline{I_i}\}_{i=1}^M$ for every $i=1,\ldots,M$. We call the set $\mathfrak{S}=\bigcup_{i=1}^M\partial I_i$ the singularity set and let $\mathfrak{S}_\infty=\bigcup_{n=0}^{\infty}f^{-n}(\mathfrak{S})$.

We define a skew-product dynamics in the following way. Let $F_i:I_i\times\R^k\mapsto I\times\R^k$ be such that
\begin{equation}\label{eq:defskew}
F_i(x,z)=(f_i(x),g_i(x,z)),
\end{equation}
where $g_i\colon I\times\R^k\mapsto\R^k$ is an affine mapping such that for every $x\in I$, the function $g_i(x,.)\colon\R^k\mapsto\R^k$ is a similitude of the form
$$
g_i(\xv,\zv)=A_i\xv+\lambda_iO_i\zv+\tv_i,
$$
where $A_i\in\R^{k\times d}$, $O_i\in O(k)$, $\lambda_i>1$ and $\tv_i\in\R^k$. We can define an uniformly expanding map $F$ on the whole region $I\times\R^k$ in the natural way,
\begin{equation}\label{eq:F}
F(x,z)=F_i(x,z)\text{ if }x\in I_i.
\end{equation}
We call the dynamical system $f\colon I\mapsto I$ the \textit{base system} of $F$.  For an example with base system on the real line, see Figure~\ref{a88b}. In case of fractal image compression, if $G\colon[0,1]^2\mapsto\R^k$ is the representation of the picture (with $k=1$ or $3$) which we want to approximate, the role of $g_i$ is to find the parameters $A_i$, $\lambda_i$, $O_i$ and $\tv_i$, for which $g_i(\xv,G(\xv))=A_i\xv+\lambda_iO_iG(\xv)+\tv_i$ fits the best to $G(f_i(\xv))$ for $\xv\in I_i$ in some sense.

\begin{figure}[h]
  \centering
  \includegraphics[width=\textwidth]{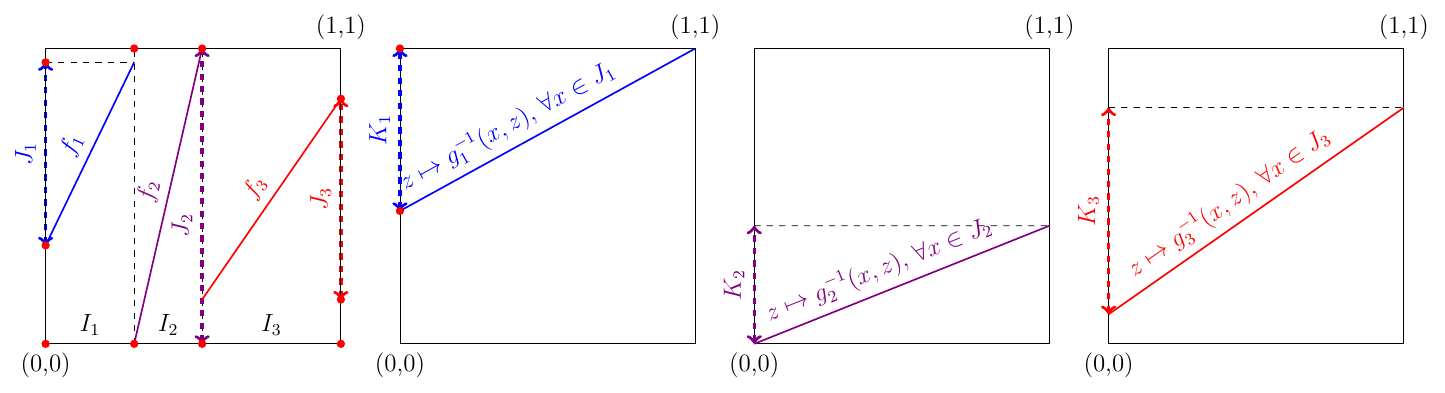}

        \begin{subfigure}[b]{0.33\textwidth}
                \centering
                \includegraphics[width=.95\linewidth]{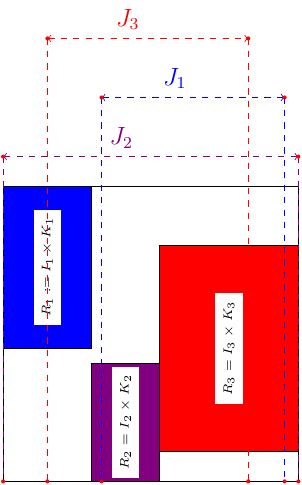}
                \caption{$R_i=F_{i}^{-1}\left(J_i\times [0,1]\right)$\\
                $\Lambda_1:=\cup_{i=1}^{3}R_i$}
                \label{a90b}
        \end{subfigure}%
        \begin{subfigure}[b]{0.33\textwidth}
                \centering
                \includegraphics[width=.95\linewidth]{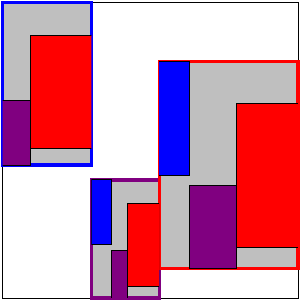}
                \caption{$\Lambda_2:=$\\
                $\cup_{i=1}^{3}
                F_{i}^{-1}\left(\left(J_i\times [0,1]\right)\cap \Lambda_1\right)
                $}
                \label{a91b}
        \end{subfigure}%
         \begin{subfigure}[b]{0.33\textwidth}
        \centering
                \includegraphics[width=.95\linewidth]{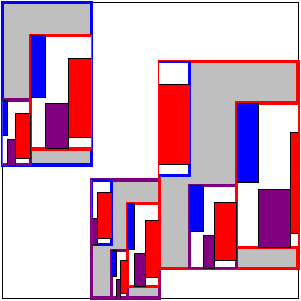}
                \caption{$\Lambda_3:=$\\
                $\cup_{i=1}^{3}
                F_{i}^{-1}\left(\left(J_i\times [0,1]\right)\cap \Lambda_2\right)
                $}
                \label{a89b}
        \end{subfigure}%
        \caption{The dynamics of $f$ and the local inverses of $F$ with non-Markovian base system.}\label{a88b}
\end{figure}

\begin{figure}[h]
  \centering
  \includegraphics[width=\textwidth]{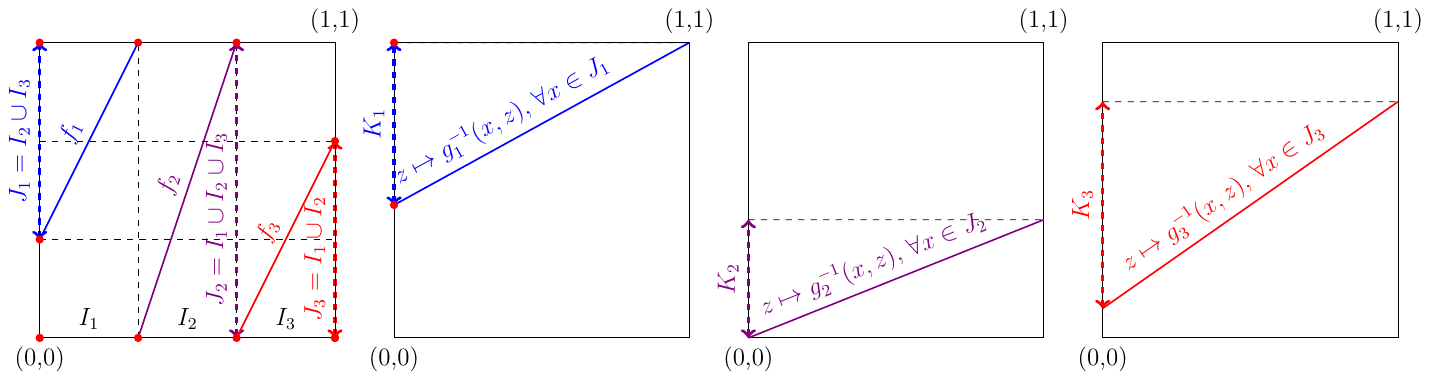}
        \begin{subfigure}[b]{0.33\textwidth}
                \centering
                \includegraphics[width=.95\linewidth]{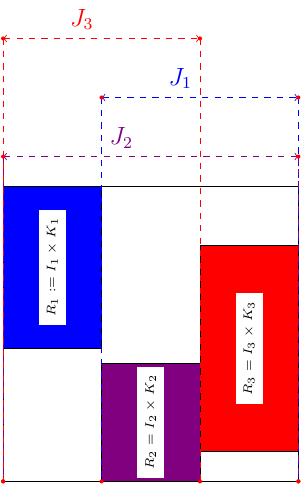}
                \caption{$R_i=F_{i}^{-1}\left(J_i\times [0,1]\right)$\\
                $\Lambda_1:=\cup_{i=1}^{3}R_i$}
                \label{a90}
        \end{subfigure}%
        \begin{subfigure}[b]{0.33\textwidth}
                \centering
                \includegraphics[width=.95\linewidth]{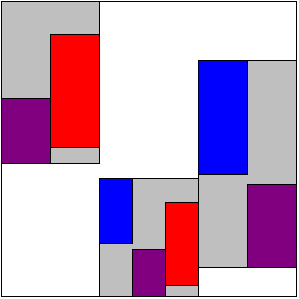}
                \caption{$\Lambda_2:=$\\
                $\cup_{i=1}^{3}
                F_{i}^{-1}\left(\left(J_i\times [0,1]\right)\cap \Lambda_1\right)
                $}
                \label{a91}
        \end{subfigure}%
         \begin{subfigure}[b]{0.33\textwidth}
        \centering
                \includegraphics[width=.95\linewidth]{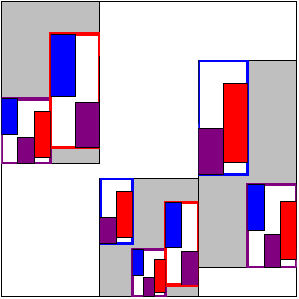}
                \caption{$\Lambda_3:=$\\
                $\cup_{i=1}^{3}
                F_{i}^{-1}\left(\left(J_i\times [0,1]\right)\cap \Lambda_2\right)
                $}
                \label{a89}
        \end{subfigure}%
        \caption{The dynamics of $f$ and the local inverses of $F$ with Markov base system.}\label{a88}
\end{figure}

By hyperbolicity, there exists open, bounded and simply connected set $U\subset\R^{d+k}$ and a uniformly contracting functions $\widetilde{F}_i$, defined on $\R^{d+k}$ such that for every $i=1,\dots,M$
	$$
	\widetilde{F}_i^{-1}(\xv)=F_i(\xv)\text{ for }\xv\in I_i\times\R^k,
	$$
	and
	\begin{equation}\label{eq:U}
	\bigcup_{i=1}^M\widetilde{F}_i(U)\subseteq U.
	\end{equation}
We call the functions $\widetilde{F}_i$ the local inverses of $F$. Thus, the system $F$ has a unique, nonempty and compact repeller $\Lambda$, for which
\begin{equation}\label{eq:lambda}
\bigcup_{i=1}^M\widetilde{F}_i(\Lambda\cap F_i(I_i\times\R^k))=\Lambda.
\end{equation}
By the construction, for every $x\in I\setminus\mathfrak{S}_\infty$ and $y\neq z\in\R$,  $\|F^n(x,y)-F^n(x,z)\|\to\infty$ as $n\to\infty$, thus $\Lambda$ is a graph of a function $G\colon I\setminus~\mathfrak{S}_\infty\mapsto~\R^k$. That is, $G$ is the function for which
$$
G(x)=z\text{, where }\{F^n(x,z)\}_{n=1}^{\infty}\text{ is bounded.}
$$

The dimension theory of non-conformal repellers (like $\Lambda$) is a widely developing topic in fractal geometry, see for example Chen and Pesin \cite{ChenPesin} and Falconer \cite{Falconersurvey}. In our setup, the dimension of $\Lambda$ can be approximated by the dimension of self-affine sets (see precise details later). Falconer \cite{falconer1988hausdorff} showed a general upper bound on the dimension in terms of the singular values, called affinity dimension. Many authors have obtained matching lower bounds in special cases. Falconer \cite{falconer1988hausdorff}, Solomyak \cite{solomyak1998measure} and Jordan, Pollicott and Simon \cite{JordanPollicottSimon07} studied the case of self-affine sets in which the translation parameters are chosen Lebesgue-typically. For planar self-affine sets, Hueter and Lalley~\cite{HueterLalley} and K\"aenm\"aki and Shmerkin \cite{KaenmakiShmerkin} showed that the dimension is equal to the affinity dimension for special classes, such as those satisfying bunching conditions. Later for planar systems, it was shown that under the assumption that the Furstenberg measure of the associated matrix random walk is sufficiently large, the Hausdorff dimension equals to the affinity dimension, see Morris and Shmerkin \cite{morris2017equality}, Rapaport \cite{rapaport2015self}, and B\'ar\'any and K\"aenm\"aki \cite{barkaen}. Most recently, B\'ar\'any, Hochman and Rapaport \cite{BHR} solved the problem on the plane under separation condition and positivity of the dimension of the Furstenberg measure. This paper highly relies on these results.

Throughout the paper, the Hausdorff dimension of a set $A$ is denoted by $\dim_HA$, and the (lower) Hausdorff dimension of a measure $\mu$ is denoted by $\dim_H\mu$ too. For the definition and properties of the Hausdorff dimension, see Falconer \cite{Falconer2} and Mattila \cite{mattila1999geometry}.

Let $\mu$ be a $F$-invariant, ergodic measure on $I\times\R^k$. Let us denote by $\chi_1(\mu)$ the Lyapunov exponent of $f$ w.r.t measure $(\proj)_*\mu$, where $\proj:I\times\R^k\mapsto I$. Moreover, let $\chi_2(\mu)$ be the Lyapunov exponent w.r.t. the skew product. That is,
\begin{eqnarray*}
&&\chi_1(\mu)=\int\log\|D_{\proj(\xv)}f\|\mathrm{d}\mu(\xv)=\sum_{i=1}^M\mu(I_i\times\R^k)\log\gamma_i\text{ and }\\
&&\chi_2(\mu)=\int\log\|\partial_2g(\xv)\|\mathrm{d}\mu(\xv)=\sum_{i=1}^M\mu(I_i\times\R^k)\log\lambda_i,
\end{eqnarray*}
where $\partial_1$ and $\partial_2$ denotes the derivative matrix w.r.t. the $x$ and $z$ coordinates respectively.

If $0<\chi_1(\mu)\leq\chi_2(\mu)$ then
$$
\dim\mu=\frac{h_{\mu}}{\chi_1(\mu)},
$$
without any further restriction. Indeed, the upper bound is trivial and the lower bound follows from the fact that $\proj_*\mu$ is $f$-invariant and ergodic and the result of Hofbauer and Raith \cite[Theorem~1]{hofbauer_raith}.

Let us define the Lyapunov dimension $D$ of an ergodic measure $\mu$ for which $\chi_1(\mu)>\chi_2(\mu)>0$ in the usual way,
\begin{equation}\label{eq:lyapdim}
D(\mu):=\min\left\{\frac{h_{\mu}}{\chi_2(\mu)},k+\frac{h_{\mu}-\chi_2(\mu)}{\chi_1(\mu)}\right\}.
\end{equation}
This definition corresponds to \cite[Definition~1.6]{JordanPollicottSimon07}. We note that for our system $D(\mu)< d+k$.

Unfortunately, our methods do not allow us to handle the case $\chi_1(\mu)>\chi_2(\mu)$ for some ergodic, invariant measure, in complete generality. Throughout the paper, we every time assume that
$$
\|D_{\mathrm{proj}(\xv)}f\|>\|\partial_2g(\xv)\|\text{ for every }\xv\in I\times\R^k,
$$
that is, the expansion is much stronger on the base system than in the second coordinate.

Let us denote the pressure function induced by the potential
\begin{equation}\label{eq:potential}
\varphi^s(\xv)=\begin{cases}
s\log\|(\partial_2g(\xv))^{-1}\| & \text{if }0\leq s\leq k, \\
k\log\|(\partial_2g(\xv))^{-1}\|+(s-k)\log\|(D_{\proj(\xv)}f)^{-1}\| & \text{if }k<s\leq d+k.
\end{cases}
\end{equation}
by $P_{\mathrm{Hof}}\colon[0,d+k)\mapsto\R$. This pressure $\phof$ is defined in the same way as the pressure given by Hofbauer~\cite[Section~3]{Hofbauerfunda}. We give the precise definition and further properties of this pressure later in Section~\ref{sec:piece}.

Finally, before we state our main results, we need a generalised version of Hochman's exponential separation condition (see \cite{barany2016dimension} and \cite{hochman2012self}), which was introduced in Hochman \cite{hochman2015self}.

\begin{definition}[HESC]\label{def:hochman}
  We say that an iterated function system (IFS) of similitudes $\{h_i\colon x\mapsto\lambda_iO_ix+t_i\}_{i=1}^M$ on $\R^k$ satisfies the {\rm Hochman's exponential separation condition} (HESC) if
  \begin{enumerate}
    \item $\limsup_{n\to\infty}\min_{\ii|_n\neq\jj|_n}\dfrac{\log\|h_{\ii|_n}(\mathbf{0})-h_{\jj|_n}(\mathbf{0})\|}{n}>-\infty$,
   \item\label{item:HESC2} the group $\mathcal{S}(\{O_i\}_{i=1}^M)$ generated by the orthogonal parts is strongly irreducible, that is, there is no finite collection $W$ of non-trivial subspaces of $\R^k$ such that $OV\in W$ for every $V\in W$ and  $O\in\mathcal{S}(\{O_i\}_{i=1}^M)$.
  \end{enumerate}
\end{definition}

Observe that part 2 of the condition is relevant only in the case when $k\geq2$. For example, if $k=1$ then the system $\{x\mapsto x/3,x\mapsto x/3+t,x\mapsto(x+2)/3\}$, where $t\notin\mathbb{Q}$, satisfies the HESC, see \cite{hochman2012self}. For more recent results see Rapaport \cite{rapaport2020}.

Now, we are ready to state the main theorems of this paper. We consider the Hausdorff dimension of $\Lambda$ in four cases. Although, the statements of the theorems are quite similar, the proofs differ significantly, thus, it is natural to state them into separate theorems. First, we discuss the case of non-Markovian 1-dimensional base.

\begin{theorem}[Diagonal, non-Markov]\label{thm:maindiag}
	Let $f$ be a piecewise linear expanding map as in \eqref{eq:base} with $d=1$. Suppose that $k=1$ and $g_i$ has the form
$$
	g_i(x,z)=g_i(z)=\lambda_iz+t_i,\text{ }\lambda_i>1, t_i\in\R,
	$$
and $\gamma_i>\lambda_i$ for every $i=1,\ldots, M$. If the IFS $\{g_i^{-1}\}_{i=1}^M$ satisfy HESC then
	$$
	\dim_H\Lambda=\sup_{\mu\in\mathcal{M}_{\rm{erg}}(\Lambda)}D(\mu)=s_0,
	$$
	where $s_0$ is the unique number such that $\phof(s_0)=0$.
\end{theorem}


We call the system $F\colon I\times\R^k\mapsto I\times\R^k$ \textit{essentially non-diagonal} if $d=k=1$ and the matrices $\{DF_i\}_{i=1}^M$ are not simultaneously diagonalisable along the dynamics. More precisely, there exists finite length words $\iv,\jv$ and $\hbar$ such that
\begin{enumerate}
	\item $\iv\neq\jv$ and $\iv\hbar\jv$ is admissible,
	\item the functions $f_{\iv}$ and $f_{\jv}$ have fixed points, and
\begin{equation}\label{eq:essnondiag}
\frac{\partial_1g_{\iv}}{f_{\iv}'-\partial_2g_{\iv}}\neq\frac{\partial_1g_{\jv}}{f_{\jv}'-\partial_2g_{\jv}}\text{ and }
\frac{(f_{\iv}'-\partial_2g_{\iv})\partial_1g_{\hbar}+f_{\hbar}'\partial_2g_{\hbar}\partial_1g_{\iv}}{f_{\hbar}'(f_{\iv}'-\partial_2g_{\iv})}\neq\frac{\partial_1g_{\jv}}{f_{\jv}'-\partial_2g_{\jv}}.
\end{equation}
\end{enumerate}
 We note that since that $f_i$ and $g_i$ are linear function, thus, the place of evaluation is redundant. In particular, \eqref{eq:essnondiag} implies that the eigenspaces corresponding to the eigenvalues $\lambda_\iv$ and $\lambda_\jv$ are different and there exists a path $\hbar$ connecting $\iv$ and $\jv$ so that the eigenspaces are not mapped into each other by the matrix $DF_{\hbar}$. Let us note that $DF_i$ is always a lower-triangular $2\times2$ matrix, but we say that we are in the triangular case when we want to emphasize that $DF_i$ is not diagonal.

\begin{theorem}[Triangular, non-Markov]\label{thm:maintriang}
	Assume that $d=k=1$ and $F$ is essentially non-diagonal and $f$ is a topologically transitive, piecewise linear expanding map. If $\gamma_i>\lambda_i$ for every $i=1,\ldots, M$ then
	$$
	\dim_H\Lambda=\sup_{\mu\in\mathcal{M}_{\rm{erg}}(\Lambda)}D(\mu)=s_0,
	$$
	where $s_0$ is the unique number such that $\phof(s_0)=0$.
\end{theorem}

We have to treat the diagonal and the triangular cases in very different ways, the proof of the diagonal case is not a special case of the triangular situation. In particular, in the triangular situation we strongly rely on the assumption that the system is essentially non-diagonal, and in the diagonal case we use heavily the property that the projections are self-similar, which is not the case in the general triangular situation.

Also, the previous results strongly rely on the work of Hofbauer \cite{hofbauerergod, Hofbauerfunda}, Hofbauer and Raith \cite{hofbauer_raith} and Hofbauer and Urba\'nski \cite{hofbauerurbanski} on piecewise monotone interval maps, which techniques allows to approximate the set $\Lambda$ with Markov-subsets.

In the next theorems, we focus on the cases when $d$ is not necessarily equals to $1$ but then we require that the base system $f$ is Markov. In this case, the pressure $\pmar$ corresponding to the potential defined in \eqref{eq:potential} is the usual pressure function defined over subshifts of finite type. For precise definition, see Section~\ref{subsec:subaddpres}.

\begin{theorem}[Diagonal, Markov]\label{thm:Markov}
	Let $f$ be a piecewise linear expanding Markov map as in \eqref{eq:base} with $d\geq1$. Suppose that $g_i$ has the form
$$
	g_i(\xv,z)=g_i(z)=\lambda_iO_iz+t_i,\text{ }\lambda_i>1, O_i\in O(k),t_i\in\R^k,
	$$
and $\gamma_i>\lambda_i$ for every $i=1,\ldots,M$. If the IFS $\{g_i^{-1}\}_{i=1}^M$ satisfies HESC then
	$$
	\dim_H\Lambda=\dim_B\Lambda=\sup_{\mu\in\mathcal{M}_{\rm{erg}}(\Lambda)}D(\mu)=s_0,
	$$
	where $s_0$ is the unique number such that $\pmar(s_0)=0$.
\end{theorem}

\begin{theorem}[Triangular,Markov]\label{thm:triangularmarkov}
	Let $f$ be a piecewise linear expanding Markov map as in \eqref{eq:base} with $d\geq1$. Suppose that $k=1$ and $g_i$ has the form
$$
	g_i(\xv,z)=\lambda_iz+A_i\xv+t_i,\text{ }\lambda_i>1, t_i\in\R,
	$$
and $\gamma_i>\lambda_i$ for every $i=1,\ldots,M$. If $F$ satisfies {\rm Rapaport's condition} then
	$$
	\dim_H\Lambda=\dim_B\Lambda=\sup_{\mu\in\mathcal{M}_{\rm{erg}}(\Lambda)}D(\mu)=s_0,
	$$
	where $s_0$ is the unique number such that $\pmar(s_0)=0$.
\end{theorem}

We will specify the Rapaport's condition later in Section~\ref{sec:tools}. In particular, Rapaport's condition holds if the Furstenberg-Kifer IFS, which can be deduced from the matrices $\{DF_i\}_{i=1}^M$ (see precise definition later in \eqref{eq:furstsys}), satisfies the HESC and $s_0$ is large (see Corollary~\ref{cor:raphoch}).

\subsection*{Organisation} In the second section, we collect all the tools we require for the proofs. Namely, in Section~\ref{sec:symb} we give the basic notations on the symbolic dynamics; in Section~\ref{subsec:subaddpres} we give the basic properties of the subadditive pressure function; in Section~\ref{sec:piece} we collect Hofbauer's results on piecewise monotone interval maps and the definition of the pressure $P_{\mathrm{Hof}}$, especially the approximation with Markov subsystems; in Section~\ref{sec:tools} we state recent results on the dimension theory of self-affine sets. Finally in Section~\ref{sec:approx}, we present a method (following Jordan and Rams) how to approximate Markov subsystems with $n$-step full shifts. After collecting all the required tools, we prove the upper bound of the Hausdorff dimension of $\Lambda$ for general $1$-dimensional systems, with non-Markovian piecewise monotone expanding interval maps in Lemma~\ref{lem:ubnonmarkov} of Section~\ref{sec3}. In Section~\ref{sec:high}, we prove the lower bound for systems with Markov base system (Theorem~\ref{thm:mainmark}), and by using this result, we prove the general base case in Section~\ref{sec5}. Finally, in Section~\ref{sec:ex}, we present some applications of our results, namely, for fractal interpolation functions (Section~\ref{sec:interpol}), for the multivariable- and the $\beta$-Takagi functions (Sections~\ref{sec:multitaka} and \ref{sec:betataka}).

\section{Preliminaries}

\subsection{Symbolic dynamics}\label{sec:symb}
In this section, we define the corresponding symbolic space to the dynamics in \eqref{eq:base} and \eqref{eq:defskew}. Let $d\geq1$ and let $\{I_i\}_{i=1}^M$ be a partition of the unit cube $[0,1]^d$ into cubes. Moreover, let $f_i:I_i\mapsto[0,1]^d$ and $f$ be defined as in Section~\ref{sec:intro}.

Let $\Sigma=\{1,\ldots,M\}^\N$ be the set of all infinite sequences of symbols $\{1,\ldots,M\}$. Denote $\Sigma^*$ the finite sequences. Let $\sigma$ be the usual left-shift operator on $\Sigma$, that is, $\sigma(i_0,i_1,\ldots)=(i_1,i_2,\ldots)$.

For a word $\ii=(i_0,i_1,\ldots)$, let $\ii|_n=(i_0,\ldots,i_{n-1})$, and for $\ii\in\Sigma^*$ let us denote the length of $\ii$ by $|\ii|$. Moreover, for a finite word $\jj=(j_0,\ldots,j_{n-1})\in\Sigma^*$ let
$[\jj]=\{\ii\in\Sigma: i_k=j_k\text{ for }k=0,\ldots,n-1\}$. For two finite or infinite words $\ii$ and $\jj$, let $\ii\wedge\jj$ denote the common part of $\ii$ and $\jj$, that is,  $\ii\wedge\jj=(k_1,\ldots,k_n)$, where $k_\ell=i_\ell=j_\ell$ for every $\ell=1,\ldots,n$ and $i_{n+1}\neq j_{n+1}$.

We note that whenever we refer to a probability measure on $\Sigma$, it is measurable with respect to the Borel $\sigma$-algebra generated by the cylinder sets. We call $Y\subseteq\Sigma$ a \emph{subshift} if it is compact w.r.t. the topology generated by the cylinder sets and $\sigma$-invariant.  For a subshift $Y$, let $$Y_n=\{\ii\in\{1,\ldots,M\}^n: [\ii]\cap Y\neq\emptyset\}.$$

We define the topological entropy,
$$
h_{top}(\sigma|_Y)=\lim_{n\to\infty}\frac{1}{n}\log\#Y_n,
$$
see \cite[Theorem~7.1 and Theorem~7.11]{Walters82}.

	Let us define the set of admissible words w.r.t. the map $f:[0,1]^d\mapsto [0,1]^d$. That is, let $X$ be the closure of the set
	\begin{equation}\label{eq:symbshift}
	\left\{(i_0,i_1,\dots)\in\Sigma:\exists x\in I\text{ such that }\forall n\geq0,\ f^n(x)\in I_{i_n}^o\right\}.
	\end{equation}
It is easy to see that $X$ is a subshift.

In order to connect the symbolic dynamics on $X$ with the dynamics of the map $f$, we define the natural projection $\pi\colon X\mapsto[0,1]^d$ so that
\begin{equation}\label{eq:natproj}
\pi(\ii)=\bigcap_{n=0}^{\infty}f^{-n}\left(\overline{I_{i_n}}\right).
\end{equation}
It is clear from the definition that $f$ and $\sigma$ are conjugated, that is, for every $\ii\in X$, $\pi\circ\sigma(\ii)=f\circ\pi(\ii)$.

We say that a subshift $Y$ is a \emph{subshift of finite type}, if there exists a finite set of forbidden words $T\subset\Sigma^*$ such that $\ii\in Y$ if and only if for every $k,n\geq0$, $(\sigma^k\ii)|_n\notin T$. We note that the set of forbidden words is not unique. We say that $Y$ is a type-$n$ subshift if $n$ is the smallest integer for which there exists a set of forbidden words such that the longest word has length at most $n+1$.

\begin{remark}\label{rem:newalpha}
We note that if $Y$ is a subshifts of type-$n$, then we can define a new alphabet $\mathcal{A}=\{1,\ldots,M^n\}$, and $\Psi\colon\{1,\ldots,M\}^n\mapsto\{1,\ldots,M^n\}$ (defined in the most natural way) and $\Psi'\colon\Sigma\mapsto\mathcal{A}^\N$ such that
for $\ii=(i_0,i_1,\ldots)$ then $\Psi'(\ii)=(\Psi(i_0,\ldots,i_n),\Psi(i_1,\ldots,i_{n+1}),\ldots)$. Moreover, there exists an $M^n\times M^n$ matrix $Q$ with elements $0$, $1$ such that
$\jj\in\Psi(Y)$ if and only if $Q_{j_\ell,j_{\ell+1}}=1$ for every $\ell=0,1,\ldots$. We call $Q$ the transition matrix.
\end{remark}

\subsection{Subadditive pressure on $\R^d$}\label{subsec:subaddpres}
Let $d\geq1$ and let $\{I_i\}_{i=1}^M$ be a partition of the unit cube $I:=[0,1]^d$ into cubes. Moreover, let $f_i:I_i\mapsto I$ and $f$ be defined as in Section~\ref{sec:intro}. We call a set $A\subset[0,1]^d$ \emph{invariant} if $f(A)\subseteq A$. Let us also define $F_i$ and $F$ as in Section~\ref{sec:intro}.

We say that a rectangle $R$ is axes parallel, if $R=R_x\times R_z$, where $R_x\subset\R^d$ and $R_z\subset\R^k$ are cubes. Recall that there exists a non-empty, open and bounded set $U$ such that $\bigcup_{i=1}^M\widetilde{F}_i(U)\subseteq U$, where $\widetilde{F}_i$ denotes the local inverse of $F_i$. Without loss of generality, we may assume that the set $U$ in \eqref{eq:U} can be chosen an axes parallel rectangle. It is easy to see that since $F_i$ are in skew product form then $\widetilde{F}_i$ has also a skew product form. That is,
	$$
	\widetilde{F}_i=(\widetilde{f}_i,\widetilde{g}_i),
	$$
	where $\widetilde{f}_i$ is a conformal, unif. contracting mapping on $\R^d$ such that
$$
f_i(x)=\widetilde{f}_i^{-1}(x)\text{ for }x\in I_i,
$$
 and $\widetilde{g}_i:\R^{d+k}\mapsto\R^k$ is an affine map such that for every $x\in I$, the mapping $\widetilde{g}_i(x,.)\colon\R^k\mapsto\R^k$ is a strictly contracting similitude for every $i=1,\ldots,M$. For the visualisation of the local inverses $\widetilde{F}_i, \widetilde{f}_i$, see Figure~\ref{a88b}.

Let $X$ be as in Section~\ref{sec:symb}. For $\ii\in\Sigma^*$, let $\widetilde{F}_{\ii}=\widetilde{F}_{i_0}\circ\cdots\circ\widetilde{F}_{i_n}$, and let $D\widetilde{F}_{\ii}$ be the linear part of the affine mapping $\widetilde{F}_{\ii}$. It is easy to see that $\widetilde{F}_{\ii}(U)$ is a parallelepiped. For $\ii\in X$, we call $\widetilde{F}_{\ii|_n}(U)$ the $n$th level cylinder set.

Similarly to $\pi$, we define the natural projection $\Pi\colon X\mapsto\R^{d+k}$ so that
$$
\Pi(\ii)=\bigcap_{n=1}^\infty\widetilde{F}_{\ii|_n}(\overline{U}).
$$

By using the symbolic expansion, we define a pressure, which is called the \emph{subadditive pressure} introduced by Falconer~\cite{falconer1988hausdorff}, which will be used in the Markov situation. First, denote $\phi^s$ the \emph{singular value function}: for a matrix $A$
$$
\phi^s(A)=\begin{cases}
\alpha_{\lceil s\rceil}(A)^{s-\lfloor s\rfloor}\prod_{j=1}^{\lfloor s\rfloor}\alpha_j(A) & \text{if }0\leq s\leq\mathrm{rank}(A), \\
|\det(A)|^{s/\mathrm{rank}(A)} & \text{if }\mathrm{rank}(A)<s,
\end{cases}
$$
where $\alpha_i(A)$ denotes the $i$th singular value of $A$. For any compact invariant set $B\subseteq[0,1]^d$, let
\begin{equation}\label{eq:pressuredefsymb}
P(s,B)=\lim_{n\to\infty}\frac{1}{n}\log\left(\sum_{\ii\in\pi^{-1}(B)_n}\phi^s\left(D\widetilde{F}_{\ii}\right)\right),
\end{equation}
where $\pi$ is the natural projection defined in \eqref{eq:natproj}, and thus, $\pi^{-1}(B)_n$ denotes all the admissible words with length $n$ in $B$. The pressure $s\mapsto P(s,B)$, defined in \eqref{eq:pressuredefsymb}, is the pressure we referred as $s\mapsto\pmar(s)$ in Theorem~\ref{thm:Markov} and Theorem~\ref{thm:triangularmarkov}.

We define the \emph{singularity dimension} over $B$ as the unique root of the equation
\begin{equation}\label{eq:modifpressure}
P(s,B)=0.
\end{equation}
Let us denote the unique root by  $s_0(B)$. The singularity dimension plays a natural role in the covering of the cylinder sets, which are ellipsoids, with balls, see Falconer~\cite{falconer1988hausdorff}. For completeness, we verify here the upper bounds. For simplicity, for a potential function $\psi\colon\Sigma\mapsto\R$ and $n\in\N$ denote $S_n\psi$ the Birkhoff sum, that is, $S_n\psi(\ii)=\sum_{k=0}^{n-1}\psi(\sigma^k\ii)$.

\begin{lemma}\label{lem:ubandadditive}
	Let $f_i$, $f$ and $F_i$, $F$ be as defined in Section~\ref{sec:intro} such that $\|D_{\proj(\xv)}f\|>\|\partial_2g(\xv)\|$ for every $\xv$. Then,
	\begin{equation}\label{eq:upperb}
	\dim_H\left\{\xv:\{F^n(\xv)\}_{n=0}^{\infty}\text{ is bounded}\right\}\leq s_0([0,1]^d),
	\end{equation}
	where $s_0([0,1]^d)$ is the root of the pressure defined in \eqref{eq:pressuredefsymb}. Moreover, for every $s>0$
\begin{equation}\label{eq:compare}
P(s,B)=\lim_{n\to\infty}\frac{1}{n}\log\left(\sum_{\ii\in\pi^{-1}(B)_n}e^{S_n\varphi^s(\pi(\ii))}\right),
\end{equation}
where $\varphi^s$ is the potential defined in \eqref{eq:potential}.
\end{lemma}

\begin{proof}
First, let us introduce an intermediate pressure.  Let $R(\ii,U)$ be the smallest closed axes parallel rectangle, which contains $F_{\ii}(U)$. Moreover, let
  $$
\phi_R^s(\ii,U):=\begin{cases}
                                                        |R(\ii,U)_z|^s, & \mbox{if } s\leq k \\
                                                        |R(\ii,U)_z|^k|R(\ii,U)_x|^{s-k}, & \mbox{if }k<s\leq d+k.
                                                      \end{cases}
 $$
For a compact invariant set $B$, let
 \begin{equation}\label{eq:pressuredefrec}
P_{R}(s,B):=\liminf_{n\to\infty}\frac{1}{n}\log\left(\sum_{\ii\in\pi^{-1}(B)_n}\phi^s_R\left(\ii,U\right)\right).
\end{equation}

Because of the skew-product structure of $F$ of conformal maps both in the base and in the fiber and because $\|D_{\proj(\xv)}f\|>\|\partial_2g(\xv)\|$ for every $i=1,\ldots,M$ and $\xv\in\overline{I_i}$, there exists a constant $c>0$ such that for every $n\geq0$ and $\ii\in X_n$,
$$
c^{-1}\phi^s(D\widetilde{F}_{\ii})\leq\phi^s_R(\ii,U)\leq c\phi^s(D\widetilde{F}_{\ii}),
$$
where the constant
$$
c=\frac{\max\limits_{\xv}\frac{\|\partial_1g(\xv)\|}{\|\partial_2g(\xv)\|}}{1-\min\limits_{\xv}\frac{\|\partial_2g(\xv)\|}{\|D_{\proj(\xv)}f\|}}.
$$
Thus,
$$
P(s,B)=P_{R}(s,B)\text{ for every }s\geq0.
$$

Observe that for every $\ii\in X$ and $n\geq1$, the cylinder set $\widetilde{F}_{\ii|_n}(\overline{U})$ can be covered by at most $\phi_R^s(\ii,U)$ many squares of side length at most $\gamma^n$, where $\gamma=\max_i\|\partial_2\widetilde{g}_i\|$. Hence,
$$
\mathcal{H}^s_{\gamma^n}(\Lambda)\leq\sum_{\ii\in X_n}\phi^s\left(D\widetilde{F}_{\ii|_n}\right).
$$
Thus, the proof of \eqref{eq:upperb} can be finished by letting $n\to\infty$.

Finally, by using again the skew product structure of $F$ and the assumption that $\|D_{\proj(\xv)}f\|>\|\partial_2g(\xv)\|$ for every $i=1,\ldots,M$ and $\xv\in\overline{I_i}$,
there exists a constant $c>0$, which can be chosen as the same constant in the previous estimate, such that for every $\ii\in X_n$ and every $n\geq1$ the ratio of the eigenvalues of $D\widetilde{F}_{\ii}$ and the side lengths of the rectangle $R(\ii,U)$ is bounded away from $0$ and infinity with $c$. In other words,
\begin{equation}\label{eq:compare2}
c^{-1}e^{S_n\varphi^s(\pi(\ii))}\leq\phi^s\left(D\widetilde{F}_{\ii|_n}\right)\leq ce^{S_n\varphi^s(\pi(\ii))},
\end{equation}
which finishes the proof.
\end{proof}

\subsection{Piecewise monotone maps}\label{sec:piece} A priori, the upper bound given in the previous section may be heavily suboptimal in the case of non-Markovian base systems. However, in our setup this is not the case. In order to present this, let us present here the basic notions and results for piecewise monotone interval maps following Hofbauer~\cite{Hofbauerfunda}.

Let $f\colon [0,1]\mapsto [0,1]$ be a uniformly hyperbolic, piecewise monotone interval map. We call a collection of connected intervals $\mathcal{D}$ a partition of $[0,1]$, if for every $I,J\in\mathcal{D}$ either $I\cap J=\emptyset$ or $I=J$, and $\bigcup_{I\in\mathcal{D}}I=[0,1]$. Let us denote by $\II=\{I_i\}$ the partition of $[0,1]$ w.r.t the monotonicity intervals of $f$. For two partitions $\DD, \DD'$ of $[0,1]$, we say that $\DD$ is \emph{finer} than $\DD'$ if for every $I\in\DD$ there exists $J\in\DD'$ such that $I\subseteq J$. We define the \emph{common refinement} $\DD\vee\DD'$ of two partitions $\DD,\DD'$ in the usual way, that is, $\DD\vee\DD'=\{I\cap J:I\in\DD,\ J\in\DD'\}$. We say that a partition $\II$ of monotonicity intervals is \emph{generating} if $\bigcup_{n=0}^{\infty}\bigcup_{I\in\II}f^{-n}(\partial I)$ is dense in $[0,1]$. Equivalently, if $\bigvee_{i=0}^\infty f^{-i}\II$ generates the Borel $\sigma$-algebra on $[0,1]$.

If $A$ is a $f$-invariant, compact set then let $h_{top}(f|_A):=h_{top}(\sigma|_{\pi^{-1}(A)})$.

Now, we introduce a special family of compact invariant sets. A compact invariant set $B$ is called \emph{Markov subset} if there exists a finite collection $\DD$ of closed intervals such that
\begin{enumerate}
  \item $J\subseteq \overline{I_i}$ for every $J\in\DD$ and some $i=1,\ldots,M$,
  \item $J_1^o\cap J_2^o=\emptyset$,
  \item $\bigcup_{J\in\DD}J\cap B=B$,
  \item either $f(J_1\cap B)\cap J_2\cap B=\emptyset$ or $ J_2\cap B\subseteq f(J_1\cap B)$ for every $J_1,J_2\in\DD$.
\end{enumerate}
We call $\DD$ the Markov partition of $B$. For a compact invariant set $A$, let us denote all the Markov subsets of $A$ by $\MM(A)$. For a Markov subset $B$ with Markov partition $\DD$, let $\DD_n$ denote the $n$th refinement of $\DD$ with respect to $f|_B$. That is,
$$
\DD_n=\left\{\bigcap_{i=0}^{n-1} f^{-i}(J_i\cap B):J_i\in\DD\right\}.
$$
Moreover, for a partition $\DD$ and an $x\in[0,1]$, denote $\DD(x)$ the unique element of $\DD$ for which $x\in\DD(x)$.

Let $\varphi\colon I\mapsto\R$ be a piecewise continuous potential function such that its continuity intervals contained in a refinement of $\II$. We define the \emph{pressure} of $\varphi$ with respect to a Markov subset $B$ such that
\begin{equation}\label{eq:markovpress}
  P(f|_B,\varphi)=\lim_{n\to\infty}\frac{1}{n}\log\sum_{J\in\DD_n}e^{\sup_{x\in J}S_n\varphi(x)}.
\end{equation}
Also, we can represent the pressure $P(f|_B,\varphi)$ in a symbolic way. Observe that $\DD$ defines a finite partition of $\pi^{-1}(B)$ w.r.t. cylinder sets.
\begin{equation}\label{eq:markovpresssymb}
  P(f|_B,\varphi)=\lim_{n\to\infty}\frac{1}{n}\log\sum_{\ii\in\pi^{-1}(B)_n}e^{\sup_{\jj\in[\ii]\cap\pi^{-1}(B)}S_n\varphi(\pi(\jj))}.
\end{equation}

We note that for a given Markov subset $B$, there are plenty of choice of the Markov partition but the value of the pressure does not depend on this choice.

\begin{proposition}\label{prop:joinmarkovs}
Let $f\colon [0,1]\mapsto [0,1]$ be a uniformly hyperbolic, piecewise monotone interval map. Let $A$ be a compact invariant, uncountable set such that $f|_A$ is topologically transitive and $h_{top}(f|_A)>0$. If $B_1,B_2\subseteq A$ are topologically transitive Markov subsets then there exists $B$ such that $B_1\bigcup B_2\subseteq B\subseteq A$ and $f|_B$ is topologically transitive.
\end{proposition}

The proof of the proposition can be found in \cite[Lemma~7 and Lemma~8]{hofbauerergod}.

We define the pressure of $\varphi$ over a compact invariant set $A$ as the supremum over all Markov subsets. That is,
\begin{equation}\label{eq:wholepres}
P(f|_A,\varphi)=\sup_{B\in\MM(A)}P(f|_B,\varphi).
\end{equation}

For a compact invariant set $A$, let $\mu$ be a probability measure such that $\mathrm{supp}(\mu)=A$. For a point $x\in[0,1]$ set
$$
\Delta_\rho(x):=\{n\geq 1:\mu(f^n(\II_n(x)))>\rho\}.
$$
Let $N_\rho(A,\mu)=\{x\in A:\#\Delta_\rho(x)<\infty\}$.

\begin{proposition}\label{prop:approxgood}
Let $f\colon [0,1]\mapsto [0,1]$ be a uniformly hyperbolic, piecewise monotone interval map. Let $A$ be a compact invariant, uncountable set such that $f|_A$ is top. transitive. Then for every $\mu$ probability measure with $\mathrm{supp}(\mu)=A$, $$\lim_{\rho\to0+}\dim_HN_\rho(A,\mu)=0.$$
\end{proposition}

The proof the proposition is the application of \cite[Lemma~14]{Hofbauerfunda} for uniformly hyperbolic, piecewise monotone maps.

 We say that a probability measure $\mu$ is \emph{$\varphi$-conformal over a compact invariant set $A$} if $\mathrm{supp}(\mu)=A$ and
 \begin{equation}\label{eq:confmeas}
 \mu(f(I))=\int_Ie^{P(f|_A,\varphi)-\varphi}d\mu\text{ for every }I\in\bigcup_{n=0}^\infty \II_n,
 \end{equation}
where $\varphi\colon I\mapsto\R$ is a piecewise continuous potential such that the continuity intervals contained in a refinement of $\II$. Since $f$ is hyperbolic, the partition $\II$ is generating. Thus, we get
 $$
 \mu(f^n(I))=\int_I e^{nP(f|_A,\varphi)-S_n\varphi}d\mu\text{ for every }I\in\bigcup_{n=0}^\infty \II_n.
 $$

\begin{theorem}\label{thm:existconf}
Let $f\colon [0,1]\mapsto [0,1]$ be a uniformly hyperbolic, piecewise monotone interval map with monotonicity intervals $\II$. Let $\varphi\colon [0,1]\mapsto\R$ be a piecewise continuous potential function such that its continuity intervals contained in a refinement of $\II$. Then for every compact invariant, uncountable set $A$, for which $f|_A$ topologically transitive, there exists a $\varphi$-conformal, non-atomic probability measure over $A$.
\end{theorem}

This theorem is a special version of \cite[Theorem~2]{hofbauerurbanski} in the uniformly hyperbolic setting. The proof of the theorem coincides with the verification on \cite[p. 118]{Hofbauerfunda}.

Throughout the paper, we usually work with the potential $\varphi^s$ defined in \eqref{eq:potential}. By reformulating \eqref{eq:compare}, we get$$P(f|_B,\varphi^s)=P(s,B)$$ for every Markov subset $B$ and $s\in[0,\infty)$, where $P(s,B)$ is defined in \eqref{eq:pressuredefsymb}.

Moreover, the pressure $s\mapsto P(f|_I,\varphi^s)$, defined in \eqref{eq:wholepres}, is the pressure we referred as $s\mapsto\phof(s)$ in Theorem~\ref{thm:maindiag} and Theorem~\ref{thm:maintriang}.

\begin{remark}\label{rem:presspectra}
	Let $B$ be a Markov subset such that $\pi^{-1}(B)$ is a subshift of type-1, and for every $I_i\cap B$ let $\xv_i\in(I_i\cap B)\times\R^k$. Then we can define a matrix $A^{(s)}$ such that
	$$
	A_{i,j}^{(s)}=\begin{cases} \|(\partial_2g(\xv_i))^{-1}\|^k\|(D_{\proj(\xv_i)}f)^{-1}\|^{s-k} & \text{if }I_j\cap B\subseteq f(I_i\cap B)\\
	0 & \text{otherwise.}\end{cases}
	$$
	Then $P(s,B)=\rho(A^{(s)})$, where $\rho(A)$ denotes the spectral radius of $A$.
	
	By Remark~\ref{rem:newalpha}, every subshifts of type-$n$ can be corresponded to a type-1 subshift by defining a new alphabet, and subdividing the monotonicity intervals into smaller intervals.
\end{remark}

Let us finish this subsection with the variational principle over Markov subsets. For a compact invariant set $B$, let us denote collection of all $\sigma$-invariant measures on $\pi^{-1}(B)$ by $\mathcal{P}_{\rm inv}(B)$, and similarly, the set of ergodic $\sigma$-invariant measures by $\mathcal{P}_{\rm erg}(B)$.

\begin{lemma}\label{lem:vari}
Let $f\colon [0,1]^d\mapsto [0,1]^d$ be a uniformly hyperbolic, piecewise linear and conformal map and let $B$ be a Markov set. Let $s_0$ be the root of the pressure $s\mapsto P(s,B)$, defined in \eqref{eq:modifpressure}. Then
$$
s_0(B)=\max_{\mu\in\mathcal{P}_{\rm erg}(B)}D(\mu),
$$
where $D(\mu)$ is the Lyapunov dimension of $\mu$.
\end{lemma}

\begin{proof} It is straightforward that
$$
\sup_{\mu\in\mathcal{P}_{\rm erg}(B)}D(\mu)\leq s_0(B).
$$
Thus, it is enough to show that there exists a measure $\mu$, for which equality holds.

However, the potential $\ii\mapsto\varphi^s(\pi(\ii))$ is piecewise constant and thus, H\"older continuous on $\pi^{-1}(B)$. Hence, by \cite[Theorem~1.2]{Bowen75} and Lemma~\ref{lem:ubandadditive}\eqref{eq:compare}, there exist a constant $C>0$ and a unique ergodic measure such that
$$
C^{-1}\leq\frac{\mu_{s}[\ii|_n]}{e^{-P(s,B)n+S_n(\ii)}}\leq C.
$$
By \cite[Theorem~1.22]{Bowen75},
\begin{equation}\label{eq:varia}
P(s,B)=h_{\mu_{s}}+\int\varphi^s(\pi(\ii))\mathrm{d}\mu_{s}(\ii)=h_{\mu_{s}}-\min\{s,d\}\chi_1(\mu_{s})-\max\{s-d,0\}\chi_2(\mu_{s}).
\end{equation}
Thus, by using the definition of $s_0(B)$, we get
$$
s_0(B)=D(\mu_{s_0(B)}).
$$
\end{proof}

\subsection{Tools for the dimension theory of self-affine sets}\label{sec:tools}
In this section, we state the results in the dimension theory of triangular self-affine iterated function systems (IFS), which we are going to use later. Let $\Phi=\{\widetilde{F}_i:\R^{d}\times\R^{k}\mapsto\R^{d}\times\R^{k}\}_{i=1}^N$ be a finite collection of contracting affine transformations such that
\begin{equation}\label{eq:gentriang}
\widetilde{F}_i(\xv,\zv)=\left(\gamma_iU_i\xv+\vv_i,\lambda_iO_i\zv+B_i\xv+\wv_i\right),
\end{equation}
where $1>\lambda_i>\gamma_i>0$, $U_i\in O(d)$, $O_i\in O(k)$, $A_i\in\R^{k\times d}$, $\vv_i\in\R^d$ and $\wv_i\in\R^k$ for every $i=1,\ldots,N$.

Denote the attractor of $\Phi$ by $\Lambda$. Moreover, for a probability vector $\pv=(p_i)_{i=1}^N$ let $\mu$ be the self-affine measure. In the study of the dimension theory of self-affine measures, the Furstenberg-Kifer measure and Ledrappier-Young formula plays an important role. In this section, we state the corresponding definitions and theorems.

First, let us define the Furstenberg-Kifer measure, which is supported in the Grassmannian manifold of $d$-dimensional subspaces of $\R^{d+k}$. Let us denote the Grassmannian manifold by $G(d,d+k)$. Let $\nu$ be a Bernoulli measure on $\Sigma$ with probability vector $\pv$. It is easy to see that in this case there are only two Lyapunov exponents $\chi_2=-\sum_ip_i\log\gamma_i>\chi_1=-\sum_{i}p_i\log\lambda_i>0$. By Oseledets' multiplicative theorem \cite[Theorem~3.4.1]{Arnold} there exists a unique measurable map $V:\Sigma\mapsto G(d,d+k)$ such that
\begin{eqnarray}
  V(\ii) &=& A_{i_0}^{-1}V(\sigma\ii)\label{eq:inv} \\
  \lim_{n\to\infty}\frac{1}{n}\log\|A_{i_n}\cdots A_{i_0}\vv\| &=& \chi_2,
\end{eqnarray}
for $\nu$-almost every $\ii\in\Sigma$ and $\vv\in V(\ii)$. We call the measure $\mu_F=V_*\nu$ {\em the Furstenberg-Kifer measure}. We show that in the case of IFS of the form \eqref{eq:gentriang}, the mapping $V:\Sigma\mapsto G(d,d+k)$ is H\"older continuous, everywhere defined mapping. We give the heuristic way to define it in the simplest case $d=k=1$, where $G(1,2)=\RP$.

For matrices $A_i$ of the form
\begin{equation}\label{eq:conj}
A_i=\left(\begin{array}{cc} \gamma_i & 0 \\ b_i & \lambda_i\end{array}\right),
\end{equation}
we have
$$
A_i^{-1}\binom{1}{x}=\frac{1}{\gamma_i}\binom{1}{\frac{\gamma_i}{\lambda_i}x-\frac{b_i}{\lambda_i}}.
$$
Since $\gamma_i/\lambda_i<1$, the IFS $\{h_i:x\mapsto\frac{\gamma_i}{\lambda_i}x-\frac{b_i}{\lambda_i}\}$ is strictly contracting, and the limit
$$
v(\ii)=\lim_{n\to\infty}h_{i_0}\circ\cdots\circ h_{i_n}(0)
$$
is well defined for every $\ii\in\Sigma$. Moreover, $v\colon\Sigma\mapsto\R$ is H\"older continuous. In other words, the action of $A_i^{-1}$ is contracting on $\RP\setminus\{\binom{0}{1}\}$ with respect to a well chosen metric. Thus, by using the invariance of $V(\ii)$ and the uniqueness, $V(\ii)=\mathrm{span}\{\binom{1}{v(\ii)}\}$.

In the general situation, the Furstenberg-Kifer measure can be associated with a self-similar measure on $\R^{dk}$. Let $E=\{V\in G(d,d+k): \dim V\cap W\geq1\}$, where $W$ is the $k$ dimensional invariant subspace w.r.t. the matrices $A_i$. That is, $W=\mathrm{span}\{\hat{\ev}_\ell\}_{\ell=d+1}^{d+k}$, where $\hat{\ev}_\ell$ is the $\ell$th element of the natural basis of $\R^{d+k}$. One can associate the set $G(d,d+k)\setminus E$ with the set
$$
\left\{\bigwedge_{\ell=1}^d\binom{\ev_{\ell}}{\xv_{\ell}}:\xv_\ell\in\R^k,\ \ell=1,\dots,d\right\}\subset \wedge^d\R^{d+k},
$$
which can be associated with $\R^{dk}$. Let $U_i^T\otimes O_i^T$ denote the usual Kronecker product of the matrices $U_i^T, O_i^T$  . That is, $U_i^T\otimes O_i^T$ is the $dk\times dk$ blockmatrix defined as
$$
U_i^T\otimes O_i^T=\left(\begin{array}{cccc}
u_{1,1}^{(i)}O_i^T & u_{2,1}^{(i)}O_i^T & \cdots & u_{d,1}^{(i)}O_i^T \\
u_{1,2}^{(i)}O_i^T & u_{2,2}^{(i)}O_i^T & \cdots & u_{d,2}^{(i)}O_i^T \\
\vdots & \vdots &\ddots & \vdots \\
u_{1,d}^{(i)}O_i^T & u_{2,d}^{(i)}O_i^T & \cdots & u_{d,d}^{(i)}O_i^T
\end{array}\right),
$$
where $U_i=(u_{m,\ell}^{(i)})_{m,\ell=1}^d$. It is easy to see that $U_i^T\otimes O_i^T$ is also an orthogonal matrix. Associated to the system defined in \eqref{eq:gentriang}, let
\begin{equation}\label{eq:furstsys}
h_i(\xvv)=\frac{\gamma_i}{\lambda_i}U_i^T\otimes O_i^T\xvv+\tvv_i,
\end{equation}
for $i=1,\dots,M$ and $\xvv\in\R^{dk}$, where
$$
\tvv_i=\frac{-1}{\lambda_i}\left(\begin{array}{c}O_i^TB_i\ev_{1} \\ \vdots \\ O_i^TB_i\ev_{d}\end{array}\right),
$$
where $\ev_\ell$ is the $\ell$th element of the natural basis of $\R^d$. We call the IFS $\{h_i\}_{i=1}^M$ the \emph{Furstenberg-Kifer IFS}. Similarly, to the previous calculations,
$$
\left(\wedge^dA_i^{-1}\right)\bigwedge_{\ell=1}^d\binom{\ev_{\ell}}{\xv_{\ell}}=\frac{1}{\gamma_i^d}\bigwedge_{\ell=1}^d\binom{\ev_{\ell}}{\zv_{\ell}},
$$
where $h_i(\xv_1^T,\ldots,\xv_d^T)^T=(\zv_1^T,\ldots,\zv_d^T)^T$. Since $\gamma_i/\lambda_i<1$, the IFS $\{h_i\}_{i=1}^M$ is strictly contracting on $\R^{dk}$, hence
$$
\left(\begin{array}{c}v_1(\ii) \\ \vdots \\ v_d(\ii)\end{array}\right)=\lim_{n\to\infty}h_{i_0}\circ\cdots h_{i_n}(\boldsymbol{0})
$$
is well defined, and by the uniqueness of $V(\ii)$, we have $V(\ii)=\mathrm{span}\left\{\binom{\ev_{\ell}}{v_{\ell}(\ii)}\right\}_{\ell=1}^d$. The measure $\mu_F=V_*\nu$ is called the Furstenberg-Kifer measure.

Let us define the orthogonal projection from $\R^{d+k}$ along a subspace $V\in G(d,d+k)$ by $\proj_V$.

\begin{theorem}\cite[Corollary~2.7]{barkaen}\label{thm:iffLyap}
Let $\Phi=\{\widetilde{F}_i\}_{i=1}^M$ be the IFS of the form \eqref{eq:gentriang} with $1>\lambda_i>\gamma_i>0$ and with SOSC. Then for every $\mu$ self-affine measure
$$
\dim\mu=D(\mu)\text{ if and only if }\dim(\proj_V)_*\mu=\min\left\{k,D(\mu)\right\}\text{ for $\mu_F$-a.e. }V.
$$
\end{theorem}

In the literature, this condition has been confirmed in the following two situations. We note that in Theorem~\ref{thm:iffLyap} we do not require that $B_i\neq0$ for some $i=1,\ldots,M$.

\begin{theorem}\cite[Proposition~6.6]{BHR}\label{thm:planar}
  Let $\Phi=\{\widetilde{F}_i\}_{i=1}^M$ be the IFS of the form \eqref{eq:gentriang} with $1>\lambda_i>\gamma_i>0$ and with SOSC. Assume that $d=k=1$. If the maps $h_i$ do not have a common fixed point then
  $$
  \dim\mu=D(\mu).
  $$
\end{theorem}

In higher dimensions, we have to add an extra condition on the Furstenberg-Kifer measure.

\begin{theorem}\label{thm:rap}\cite[Section~1.2]{rapaport2015self}
  Let $\Phi=\{\widetilde{F}_i\}_{i=1}^M$ be the IFS of the form \eqref{eq:gentriang} with $1>\lambda_i>\gamma_i>0$ and with SOSC. If
  $$
  D(\mu)+\dim\mu_F>(d+1)k
  $$
  then
  $$
  \dim\mu=D(\mu).
  $$
\end{theorem}

In general, the dimension theory of the Furstenberg-Kifer measure is far from being well understood. For the case of general $SL_2(\R)$ matrices, Hochman and Solomyak \cite{Hochman2017} gave a condition, which allows us to calculate the dimension of the measure. However, in higher dimension, it is unknown whether the Furstenberg-Kifer measure is exact dimensional. In our case, the Furstenberg-Kifer measure can be associated with a self-similar measure, thus, by using the result of Hochman \cite{hochman2015self}, we can compute the dimension of the measure under some conditions.

\begin{theorem}\cite[Corollary~1.6]{hochman2015self}\label{thm:hochman}
Let $\{h_i:\R^{dk}\mapsto\R^{dk}\}_{i=1}^M$ be a IFS of similarities of the form \eqref{eq:furstsys}, and let $\pv=(p_i)_{i=1}^M$ be a probability vector with $p_i>0$ for every $i=1,\ldots,M$. If HESC holds (see Definition~\ref{def:hochman}) and $\mu_F=V_*\pv^\N$ then
$$
\dim\mu_F=\min\left\{dk,\frac{-\sum_{i=1}^Mp_i\log p_i}{-\sum_{i=1}^Mp_i\log(\gamma_i/\lambda_i)}\right\}.
$$
\end{theorem}

\begin{corollary}\label{cor:raphoch} Let $\Phi=\{\widetilde{F}_i\}_{i=1}^M$ be the IFS of the form \eqref{eq:gentriang} with $1>\lambda_i>\gamma_i>0$ and with SOSC. Let $s$ be the unique root of the pressure defined in \eqref{eq:pressuredefsymb}. If the IFS $\{h_i:\R^{dk}\mapsto\R^{dk}\}_{i=1}^M$ defined in \eqref{eq:furstsys} satisfies HESC,
$$
s>\frac{(d+2)k}{2}
$$
then
$$
\dim_H\Lambda=s.
$$
\end{corollary}

\begin{proof}
  Observe that if $s$ is the unique root of the pressure defined in \eqref{eq:pressuredefsymb},
  $$
  \text{ if $s<k$ then }\sum_{i=1}^M\lambda_i^s=1\text{ else }\sum_{i=1}^M\lambda_i^k\gamma_i^{s-k}=1.
  $$
  Observe that our assumption implies  $s>k$. Let $\nu$ be the Bernoulli measure associated to the prob vector $\pv=(\lambda_i^k\gamma_i^{s-k})_{i=1}^M$ and let $\mu_F=V_*\nu$ and $\mu=\Pi_*\nu$. Thus,
  $$
  D(\mu)=s\text{ and }\frac{-\sum_{i=1}^Mp_i\log p_i}{-\sum_{i=1}^Mp_i\log(\gamma_i/\lambda_i)}\geq s-k.
  $$
  Hence, by Theorem~\ref{thm:hochman}, $\dim\mu_F>s-k$, and therefore $D(\mu)+\dim\mu_F\geq2s-k>(d+1)k$. Thus, by applying Theorem~\ref{thm:rap}, the statement follows.
\end{proof}

\begin{remark}\label{rem:tosurface}
We note that the condition $s>(d+1)k/2$ given in Corollary~\ref{cor:raphoch} holds if $k=1$ and $d\geq2$. For $k\geq 2$, we have $(d+2)k/2\geq d+k\geq s$. On the other hand, in our case $s>d\geq d/2+1$.
\end{remark}

\subsection{Approximating Markov systems with IFSs}\label{sec:approx}
In this section, we approximate the subshifts $Y$ of finite type with full-shifts in the sense of entropy plus weak*-topology. By Remark~\ref{rem:newalpha}, we may assume that $Y$ is of type-$1$, that is, $Y$ is a Markov-shift. Throughout the section, we use the method given by Jordan and Rams \cite{JordanRams11}.

Let $Q$ be the $M\times M$ transition matrix corresponding to the subshift $Y$ of type-$1$. Let us denote the set of allowed words of length $q$ by $\Sigma^{(q)}_Q$. We say that a measure $\mu$ is Markov if there exists an $M\times M$ stochastic matrix $P$ such that if $P_{i,j}\neq0$ then $Q_{i,j}\neq0$ and
$$\mu([i_0,\dots,i_n])=p_{i_0}P_{i_0,i_1}\cdots P_{i_{n-1}i_n},$$
where $(p_i)_{i=1}^M$ is a left-eigenvector of $P$ of eigenvalue $1$. We say that $\mu$ is generalized Markov, if there exist $q\geq1$ and $M^q\times M^q$ stochastic matrix $P$ such that if $P_{\iv_1,\iv_2}\neq0$ then $\iv_1,\iv_2\in\Sigma^{(q)}_Q$ and $(\iv_1)_q=(\iv_2)_1$, moreover,
$$
\mu([i_0,\dots,i_{qn}])=\mu([\iv_0,\dots,\iv_n])=p_{\iv_0}P_{\iv_0,\iv_1}\cdots P_{\iv_{n-1}\iv_n},
$$
where $(p_{\iv})_{\iv\in\Sigma_Q^{(q)}}$ is a left eigenvector of $P$ with eigenvalue $1$. We note that generalised Markov measures are not necessarily $\sigma$-invariant, but they are $\sigma^q$-invariant. By taking $\mu'=\frac{1}{q}\sum_{k=0}^{q-1}\mu\circ\sigma^{-k}$, one can show that $\mu'$ is $\sigma$-invariant.

We say that a symbol $j\in\{1,\dots,M\}$ is recurrent if there exist $n\geq1$  and $(i_1,\dots,i_n)$ that $Q_{j,i_1}Q_{i_1,i_2}\cdots Q_{i_n,j}\neq0$. Denote $\mathcal{R}_Q$ the set of recurrent symbols. Let us define a new alphabet for all $j\in\mathcal{R}_Q$. Namely,
\begin{equation}\label{eq:omega}
\Omega_{j,Q}^{(q)}=\left\{\iv\in\Sigma_Q^{(q)}:i_0=i_q=j\right\}.
\end{equation}
Each element of $\Omega_{j,Q}^{(q)}$ corresponds to a $q$-step loop with source and target $j$ in our Markov system, and we may concatenate such loops. Let us denote the set of such infinite words by $L_j^{(q)}$. In that way we obtain a $\sigma^q$-invariant subset, which we can identify with $\Omega:=\left(\Omega_{j,Q}^{(q)}\right)^{\N}$, that is $(\Omega,\widetilde{\sigma})$ is conjugated to $(L_j^{(q)},\sigma^q)$. Let us denote the conjugation by $\varphi_q:L_j^{(q)}\mapsto\Omega$. We can define Bernoulli measures on it by attaching to each $\omega\in\Omega_{j,Q}^{(q)}$ a probabilistic weight $p_\omega$. Denote this Bernoulli measure by $\widetilde{\mu}$. The measure $(\varphi_q)_*\widetilde{\mu}$ is only $\sigma^q$-invariant and ergodic, to make it $\sigma$-invariant and ergodic, we need to consider
\begin{equation}\label{eq:qstepBernoulli}
\mu:=\frac{1}{q}\sum_{k=0}^{q-1}(\varphi_q)_*\widetilde{\mu}\circ\sigma^{-k}.
\end{equation}
We call the measure $\mu$ as $q$-step Bernoulli measure for the recurrent element $j$. Let us denote the set of $q$-step Bernoulli measures for the recurrent symbol $j$ by $\mathcal{B}_{j,q}$. Moreover, let
$$
\mathcal{B}_Q=\bigcup_{j\in\mathcal{R}_Q}\bigcup_{q=1}^{\infty}\mathcal{B}_{j,q}
$$

Now, we state a modified version of Bernoulli approximation, proven in \cite[Lemma~6]{JordanRams11}, for Markov systems. We say that a sequence $\mu_n\in\mathcal{P}_{\rm{inv}}(Q)$ converges to $\mu\in\mathcal{P}_{\rm{inv}}(Q)$ in \textit{the entropy plus weak*-topology} if $\mu_n$ converges to $\mu$ in weak*-topology and $h_{\mu_n}\to h_{\mu}$ as $n\to\infty$. We note that the entropy plus weak*-topology is indeed a topology, since $Y$ is a separable metric space and thus, the weak*-topology is metrisable with metric $d_{w*}$, and $d(\nu,\mu)=d_{w*}(\nu,\mu)+|h_\nu-h_\mu|$ is a metric generating the entropy plus weak*-topology.

\begin{lemma}\label{lem:qstepdense}
        For any primitive matrix $Q$, $\mathcal{B}_Q$ is dense in $\mathcal{P}_{\rm{inv}}(Q)$ in the entropy plus weak*-topology. If $Q$ is not primitive, $\mathcal{B}_Q$ is dense in $\mathcal{P}_{\rm{erg}}(Q)$ and its convex hull is dense in $\mathcal{P}_{\rm{inv}}(Q)$.
\end{lemma}

\begin{proof}
It is enough to prove the first statement, the second is an easy corollary.

        Thus, assume that $Q$ is primitive and let $k\geq1$ be such that all elements of $Q^k$ are positive. Let $\mu$ be an arbitrary invariant measure. Choose $q>2k$. Let $i\in\mathcal{R}_Q$ be arbitrary but fixed. Let us define a $q$-step Bernoulli measure $\nu_q$ for $i$ as follows: for any $\iv\in\Omega_{j,Q}^{(q)}$ we decompose $\iv=\iv_1\jv\iv_2$ for which $|\iv_1|=|\iv_2|=k$.

It is easy to see by the positivity of $Q^k$ that there exist $\iv_1,\iv_2$ with length $k$ such that for every $\jv\in\Sigma_Q^{(q-2k)}$, $\iv_1\jv\iv_2\in\Omega_{j,Q}^{(q)}$. Let
        $$
        \widetilde{\nu}_q(\iv):=\begin{cases}
\mu(\jv) & \iv=\iv_1\jv\iv_2 \\
0 & \text{otherwise.}
\end{cases}
        $$
        Let $\nu_q$ be as defined in \eqref{eq:qstepBernoulli}. First, observe that
        \begin{equation}\label{eq:lem11}
        h_{\nu_q}=\frac{1}{q}h_{(\varphi_q)_*\widetilde{\nu}_q}=\frac{1}{q}h_{\widetilde{\nu}_q}=\frac{-1}{q}\sum_{\jv\in\Sigma_Q^{(q-2k)}}\mu([\jv])\log\mu([\jv]).
        \end{equation}
That is, the entropies of measures $\nu_q$ converge to $h(\mu)$ as $q\to\infty$.

Now we show that $\nu_q\to\mu$ in weak*-topology. Let $\eta:Y\mapsto\R$ be a H\"older-continuous test function. That is, there exists a constant $0<\kappa<1$ such that for every $\ii,\jj\in Y$
$$
|\eta(\ii)-\eta(\jj)|\leq \kappa^{|\ii\wedge\jj|}.
$$
We have
        \begin{eqnarray*}
        &&\left|\int\eta(\ii)\mathrm{d}\nu_q(\ii)-\int\eta(\ii)\mathrm{d}\mu(\ii)\right|\\
        &&\leq\frac{1}{q}\sum_{n=0}^{q-1}\left|\int_{\Omega}\eta(\sigma^n\varphi_q(\ii))\mathrm{d}\widetilde{\nu}_q(\ii)-\int\eta(\sigma^n\ii)\mathrm{d}\mu(\ii)\right|\\
        &&\leq\frac{2k}{q}\sup_{\ii\in Y}|\eta(\ii)|+\frac{1}{q}\sum_{n=k}^{q-k}\left|\int_{\Omega}\eta(\sigma^n\varphi_q(\ii))\mathrm{d}\widetilde{\nu}_q(\ii)-\int\eta(\sigma^n\ii)\mathrm{d}\mu(\ii)\right|\\
&&\leq\frac{2k}{q}\max_{\ii\in Y}|\eta(\ii)|+\frac{2}{q}\sum_{n=k}^{q-k}\kappa^{q-k-n}+\\
        &&\frac{1}{q}\sum_{n=k}^{q-k}\left|\sum_{\jv\in\Sigma^{(q-2k)}_Q}\eta(\sigma^{n-k}\jv\jj)\mu([\jv])-\sum_{\jv\in\Sigma^{(q-2k)}_Q}\eta(\sigma^{n-k}\jv\jj)\mu([\jv])\right|\\
        &&=\frac{2k}{q}\sup_{\ii\in Y}|\eta(\ii)|+\frac{2}{q}\sum_{n=k}^{q-k}\kappa^{q-k-n}.
        \end{eqnarray*}
Thus, $\int\eta\mathrm{d}\nu_q\to\int\eta\mathrm{d}\mu$ as $q\to\infty$. Since this holds for every H\"older-continuous test function, $\nu_q\to\mu$ in weak*-topology.
                \end{proof}

\section{Upper bound for the general case with one dimensional base }\label{sec3}

In this section, we give a more sophisticated upper bound for the Hausdorff dimension of the repeller $\Lambda$ of the system $F$, defined in \eqref{eq:F}, in the case when $d=1$. Let us recall some definitions.

Let $\II=\{I_i\}_{i=1}^M$ be a finite partition of the unit interval $[0,1]$ into proper intervals. Moreover, for every $i=1,\dots,M$ let $f_i$ be a uniformly expanding similitude such that $f_i\colon I_i\mapsto [0,1]$, and we consider the uniformly, piecewise expanding dynamical system $f\colon [0,1]\mapsto [0,1]$, where
\begin{equation}\label{eq:base2}
f(x)=f_i(x)\text{ if }x\in I_i.
\end{equation}
We denote the $n$th refinement of the partition $\II$ w.r.t. $f$ by $\II_n$.

We define $F\colon [0,1]\times\R\mapsto [0,1]\times\R$ as
\begin{equation}\label{eq:defskew}
F(x,z)=(f_i(x),g_i(x,z))\text{ if }x\in I_i,
\end{equation}
where $g_i\colon [0,1]\times\R\mapsto\R$ is an affine mapping such that for every $x\in [0,1]$, the function $g_i(x,.)\colon\R\mapsto\R$ is a similitude and
$$
|f_i'(x)|>\|\partial_2g_i(x,z)\|>\lambda>1\text{ for every }(x,z)\in [0,1]\times\R\text{ and }i=1,\ldots,M.
$$
Denote the local inverses of $f$ and $F$ by $\widetilde{f}_i$ and $\widetilde{F}_i$ as in Section~\ref{subsec:subaddpres}. We may assume without loss of generality, that there exists a closed and bounded interval $J\subset\R$ such that $\widetilde{F}_i([0,1]\times J)\subseteq [0,1]\times J$.

Let $s_0$ be the unique root of the pressure $P(f,\varphi^s)=0$, where $\varphi^s$ is defined in \eqref{eq:potential} and $P$ is defined in \eqref{eq:wholepres}. Let $s>s_0$ and let $\mu$ be a $\varphi^s$-conformal measure on $[0,1]$, that is,
$$
 \mu(f^n(I))=\int_I e^{nP-S_n\varphi^s}d\mu\text{ for every }I\in\bigcup_{n=0}^\infty \II_n.
 $$
 By Theorem~\ref{thm:existconf}, there exists such measure $\mu$. For $\rho>0$, let
 $$
 G_\rho(n)=\left\{I\in\II_n:\mu(f^n(I))>\rho\right\}.
 $$
 Moreover, let
 $$
 M_\rho=\bigcap_{N=1}^{\infty}\bigcup_{n=N}^{\infty}\bigcup_{I\in G_\rho(n)}I.
 $$

 \begin{lemma}
 $$
 \lim_{\rho\to0}\dim_H\Lambda\setminus(M_{\rho}\times\R)\leq1.
 $$
 \end{lemma}

\begin{proof}
Since $\Lambda\setminus(M_{\rho}\times\R)\subseteq ([0,1]\setminus M_{\rho})\times\R$, the statement follows by Proposition~\ref{prop:approxgood}.
\end{proof}

\begin{lemma}\label{lem:ubnonmarkov}
$$
\dim_H\Lambda\leq\max\{1,s_0\},
$$
where $s_0$ is the unique root of the pressure defined in \eqref{eq:wholepres}.
\end{lemma}

\begin{proof}
For every $\rho>0$ we have $$\dim_H\Lambda\leq\max\{\dim_H\Lambda\setminus(M_{\rho}\times\R),\dim_H\Lambda\cap(M_{\rho}\times\R)\}.$$ Thus, it is enough to show that $$\lim_{\rho\to\infty}\dim_H\Lambda\cap(M_{\rho}\times\R)\leq s_0.$$

Observe that the Birkhoff sum $S_n\varphi^{s}$ is constant over the intervals in $\II_n$. So with a slight abuse of notation, we write $S_n\varphi^{s}:\II_n\to\R$ for every $s>0$.

Let $s>s_0$ and $\mu$ be the $\varphi^s$-conformal measure. We note that in this case, $P=P(f,\varphi^s)<0$. Since $\mu$ is non-atomic and compactly supported, we get that there exists $\kappa=\kappa(\rho)$ such that $|f^n(I)|>\kappa$ for every $n\ge1$ and $I\in G_\rho(n)$. Hence, by the piecewise linearity of $f$
$$
\kappa e^{-S_n\log|f'|}\leq |I|\leq e^{-S_n\log|f'|}.
$$
Let $\ii\in X_n$ the word which corresponds to $f^n(I)$. Since $|f_i'|>|\partial_2g_i|$, then $\widetilde{F}_{\ii}([0,1]\times J)$ can be covered by $(|J|+1)\cdot e^{-S_n\log|\partial_2g|+S_n\log|f'|}$ many balls with radius $e^{-S_n\log|f'|}$. Hence, the impact of $\widetilde{F}_{\ii}([0,1]\times J)$ in the $s$-dimensional Hausdorff measure is at most $e^{S_n\varphi^{s}(I)}$. Thus,
$$
\mathcal{H}^{s}(M_\rho\times J\cap\Lambda)\leq C\lim_{N\to\infty}\sum_{n=N}^{\infty}\sum_{I\in G_\rho(n)}e^{S_n\varphi^{s}(I)}.
$$

But by using the $\varphi^{s}$ conformality of the measure $\mu$,
$$
\mu(f^n(I))=e^{nP-S_n\varphi^{s}(I)}\mu(I)
$$
and hence,
\begin{eqnarray*}
\mathcal{H}^{s}(M_\rho\times J\cap\Lambda)&\leq& C\lim_{N\to\infty}\sum_{n=N}^{\infty}\sum_{I\in G_\rho(n)}\frac{e^{nP}\mu(I)}{\mu(f^n(I))}\\
&\leq& C\lim_{N\to\infty}\sum_{n=N}^{\infty}\sum_{I\in G_\rho(n)}\frac{e^{nP}\mu(I)}{\rho}\\
&\leq& C\lim_{N\to\infty}\sum_{n=N}^{\infty}\frac{e^{nP}}{\rho}=0 .
\end{eqnarray*}

Since $s>s_0$ was arbitrary, the statement follows.
\end{proof}

\begin{lemma}\label{fact}
Let $f\colon [0,1]\mapsto [0,1]$ be a uniformly hyperbolic, piecewise $C^{1+\alpha}$, piecewise monotone interval map with a finite set $\II$ of monotonicity intervals. Moreover, suppose that for each $I\in\II$, there exists an open interval $\overline{I}\subseteq J$ such that $f|_I$ can be extended to a $C^{1+\alpha}$ map to $J$.  Then there exists a constant $K>1$ such that for every $n\geq1$, every $I\in\II_n$ and for every $x,y\in I$,
$$
K^{-1}<\frac{|(f^{n})'(x)|}{|(f^{n})'(y)|}<K.
$$
\end{lemma}

Indeed, the branch of $f^{-n}|_{f^{n}(I)}$ is a composition of contracting uniformly $C^{1+\alpha}$ maps. Hence, by \cite[Proposition~20.1(1)]{Pesin97}, it has distortion uniformly bounded by a constant depending only on $\alpha$, the H\"older-constant and the uniform contraction ratio.

\begin{lemma}\label{lem:bounds0}
Let $f\colon [0,1]\mapsto [0,1]$ be a uniformly hyperbolic, piecewise $C^{1+\alpha}$, piecewise monotone interval map with a finite set $\II$ of monotonicity intervals. Moreover, suppose that for each $I\in\II$, there exists an open interval $\overline{I}\subseteq J$ such that $f|_I$ can be extended to a $C^{1+\alpha}$ map to $J$. Let $\tilde{s}_0$ be the unique root of the pressure defined in \eqref{eq:wholepres} with respect to the potential $-s\log|f'|$. Then $1=\tilde{s}_0$.
\end{lemma}

\begin{proof}
To prove the claim of the lemma, it is enough to construct a Markov-subsystem $B$ such that the root of $P(f|_B,-s\log|f'|)=0$ is arbitrary close to one. In order to construct such a system, we apply a modification of the construction in Hofbauer, Raith and Simon \cite{hofraithsim}. We call the set $\mathfrak{S}_k=\bigcup_{n=0}^{k}f^{-n}(\mathfrak{S})$, where $\mathfrak{S}=\bigcup_{I\in\II}\partial I$.

First, let us fix $N$ large. Since $f$ is uniformly hyperbolic, by taking a sufficiently high $k$ we may assume that $|(f^k)'(x)|>N$ for every $x\in[0,1]\setminus\mathfrak{S}_k$. By subdividing the intervals in $\II_k$ into smaller pieces, we can define a partition $\JJ$ refinement of $\II_k$ such that $|I_1|/|I_2|<2$ for every $I_1\neq I_2\in\JJ$.

For every $I\in\JJ$ we define a subinterval $J\subset I$ to be the maximal interval such that $f^k(J)$ is a union of intervals contained in $\JJ$. It is easy to see that $f^k(I)\setminus f^k(J)$ consists of at most two intervals, both contained in intervals in $\JJ$. Since $|f^k(I)|>N|I|$, $f^k(J)$ is formed by at least $N/2-2$ many intervals from $\JJ$. Moreover,
$$
\frac{|f^k(I)\setminus f^k(J)|}{|f^k(I)|}\leq\frac{ 4}{N}.
$$
Let $\mathcal{K}$ denote the set of intervals $J$ defined above. Let
$$
B_k=\{x\in[0,1]:\text{ for every }\ell\geq0\text{ there exists }J\in\mathcal{K}\text{ such that }f^{\ell k}(x)\in J \}.
$$
It is easy to see that $B_k$ is a Markov subset for $f^k$ and hence, $\bigcup_{m=0}^{k-1}f^m(B_k)$ is a Markov subset for $f$. Denote $\mathcal{K}_n$ the $n$th refinement of the intervals in $\mathcal{K}$ by the map $f^k|_{B_k}$.

Let $K>1$ be the distortion constant from Lemma~\ref{fact}. So, for every $n\geq1$
\begin{enumerate}
  \item $\mathcal{L}(\bigcup_{J\in \mathcal{K}_{n}}J)>\left(1-\frac{4K}{N}\right)^n$,
  \item $|J|\leq \left(\frac{1}{N}\right)^n$ for every $J\in\mathcal{K}_{n}$.
\end{enumerate}
By Lemma~\ref{fact}, for $s=\frac{\log(N-4K)}{\log N}$ and for all $n\geq1$
\begin{eqnarray*}
K\sum_{J\in\mathcal{K}_n}\max_{x\in J}e^{-S_n\log|(f^k)'(x)|}&\geq&\sum_{J\in\mathcal{K}_n}|J|^s=\sum_{J\in\mathcal{K}_n}|J|\cdot|J|^{s-1}\\
&\geq& \left(1-\frac{4K}{N}\right)^n\left(\frac{1}{N}\right)^{(s-1)n}\geq1.
\end{eqnarray*}
Hence, $P(f|_{B_k},-s\log|f'|)\geq 0$. Thus, by taking $N\to\infty$, $s\to1$ and we get $\tilde{s}_0\geq1$. The upper bound for $\tilde{s}_0$ is obvious.
\end{proof}

We note that an immediate corollary of Lemma~\ref{lem:bounds0} is
\begin{equation}\label{eq:phipot}
1\leq s_0,
\end{equation}
where $s_0$ is the unique root of the pressure defined in \eqref{eq:wholepres} with respect to the potential $\varphi^s$. Indeed, $P(f|_B,\varphi^s)\geq P(f|_B,-s\log|f'|)$, since $\varphi^s\geq-s\log|f'|$.

\section{conformal base with Markov structure}\label{sec:high}
Let $f\colon[0,1]^d\mapsto[0,1]^d$ be as in \eqref{eq:base} Markov with $d\geq1$ such that $f_i$ are similitudes. Let $F_i\colon I_i\times\R\mapsto[0,1]^d\times\R$ be such that
\begin{equation*}\label{eq:defskew}
F_i(x,z)=(f_i(x),g_i(x,z)),
\end{equation*}
where $g_i\colon I\times\R\mapsto\R$ is an affine mapping such that for every $x\in[0,1]^d$, the function $g_i(x,.)\colon\R\mapsto\R$ is a similitude and
\begin{equation}\label{eq:this}
\|Df_i\|>|\partial_2g_i|>1.
\end{equation}

\begin{theorem}\label{thm:mainmark}
  Let us assume that one of the following conditions holds.
  \begin{enumerate}[(i)]
    \item\label{it:mark1} The functions $g_i$ has the form
$$
	g_i(x,z)=g_i(z)=\lambda_iz+t_i,\text{ }|\lambda_i|>1, t_i\in\R
	$$
for every $i=1,\ldots,M$, and the IFS $\{g_i^{-1}\}_{i=1}^M$ satisfies HESC (see Definition~\ref{def:hochman});
    \item\label{it:mark2} $d\geq2$, and $\{h_i\}_{i=1}^M$ (defined in \eqref{eq:furstsys}) satisfies HESC;
    \item\label{it:mark3} $d=1$, $F$ is essentially non-diagonal and the base system $f$ is topologically transitive.
  \end{enumerate}
  If \eqref{eq:this} holds then
  $$
  \dim_H\Lambda=s_0,
  $$
  where $s_0$ is the root of the pressure defined in  \eqref{eq:compare}.
\end{theorem}

\begin{proof}
In all the three cases, the upper bound follows by Lemma~\ref{lem:ubandadditive}.

For the lower bound, it is enough to show that there exists an IFS, which attractor is contained in $\Lambda$ and has dimension arbitrary close to $s_0$. As a combination of Lemma~\ref{lem:vari} and Lemma~\ref{lem:qstepdense},
for every $\varepsilon>0$ there exists an IFS $\Phi$ with attractor $\Lambda'$ and invariant measure $\mu$ such that
$$
D(\mu)>s_0-\varepsilon,\ \Lambda\supseteq\Lambda'\text{ and }\dim_H\Lambda\geq\dim_H\mu.
$$
Moreover, the functions of $\Phi$ are finite compositions of the local inverses $\widetilde{F}_i$ of $F_i$. Hence, it is enough to prove that $\dim_H\mu=D(\mu)$.

By Theorem~\ref{thm:iffLyap}, it is enough to show that
$$
\dim(\proj_V)_*\mu=\min\left\{1,D(\mu)\right\}\text{ for $\mu_F$-a.e. }V,
$$
where $\mu_F$ is the Furstenberg-Kifer measure corresponding to $\mu$ and $\Phi$, defined in Section~\ref{sec:tools}.

It is easy to see that if the IFS $\{g_i^{-1}\}_{i=1}^M$ or $\{h_i\}_{i=1}^M$ satisfies the HESC then every finite subsystem, which is formed by finite compositions of the functions, remains to satisfy the HESC.

Hence, in case \eqref{it:mark1}, $\mu_F$ supported on one point of $G(d,d+1)$ and the claim $\dim(\proj_V)_*\mu=\min\left\{1,D(\mu)\right\}$ follows by \cite[Theorem~1.1]{hochman2012self}, and the case \eqref{it:mark2} follows by Corollary~\ref{cor:raphoch}.

Finally, we turn to the case \eqref{it:mark3}. In order to show this we use Theorem~\ref{thm:planar}. So, it is enough to show the following claim holds, which is the remaining part of the proof.

\begin{claim}
If \eqref{it:mark3} holds then the IFS $\Phi$ can be chosen such that the corresponding Furstenberg system  $\{h_{\iv}\}_{\iv: \widetilde{F}_{\iv}\in\Phi}$ has at least two distinct fixed points.
\end{claim}

By the construction of $\Phi$ in Lemma~\ref{lem:qstepdense}, there exist finite words $\iv_1$ and $\iv_2$ such that $\widetilde{F}_{\iv}\in\Phi$ if and only if $\iv=\iv_1\jv\iv_2$, where $\iv_1\jv\iv_2$ is admissible and $|\jv|=n$ with some large $n$ chosen according to the precision of the approximation. Let us argue by contradiction. That is, assume that $A_{\iv_1}A_{\jv}A_{\iv_2}$ can be simultaneously diagonalised for every $n\geq1$ and every $|\jv|=n$, for which $\iv_1\jv\iv_2$ is admissible. That is, there exists $Q\in GL_{2}(\R)$ such that $Q$ is triangular matrix with diagonal entries $1$ and
$$
QA_{\iv_1}A_{\jv}A_{\iv_2}Q^{-1}=D_{\iv_1\jv\iv_2},
$$
where $D_{\iv_1\jv\iv_2}$ is a diagonal matrix with the diagonal elements of $A_{\iv_1}A_{\jv}A_{\iv_2}$.

By the essentially non-diagonal property of $F$, there exist $\hbar_1, \hbar_2$ and $\overline{a}$ finite admissible words so that $f_{\hbar_1}$ and $f_{\hbar_2}$ have fixed points, $A_{\hbar_1}$ and $A_{\hbar_2}$ are not simultaneously diagonalisable, $\hbar_1\overline{a}\hbar_2$ is admissible and $A_{\overline{a}}v_{\hbar_2}\neq v_{\hbar_1}$, where  $v_{\hbar_1}$ and $v_{\hbar_2}$ denote the eigenvector of the matrix $A_{\hbar_1}$ and $\vv_{\hbar_2}$ respectively, different from $(0,1)$ and with first coordinate $1$. (see \eqref{eq:essnondiag}). Let $R_{\hbar_i}$ be the triangular matrix, with diagonal entries $1$ so that $R_{\hbar_i}^{-1}A_{\hbar_i}R_{\hbar_i}=D_{\hbar_i}$. That is $R_{\hbar_i}=[v_{\hbar_i},(0,1)^T]$.

Since $f$ is topologically transitive, there exists $\jv_1, \jv_2$ and $\jv_1', \jv_2'$ such that $\iv_1\jv_1\hbar_1$, $\hbar_1\jv_2\iv_2$, $\iv_1\jv_1'\hbar_2$ and $\hbar_2\jv_2'\iv_2$ are admissible. By the fixed point property, $\iv_1\jv_1\hbar_1\hbar_1\jv_2\iv_2$ is admissible too. Hence,
\begin{eqnarray*}
D_{\iv_1\jv_1\hbar_1\hbar_1\jv_2\iv_2}&&=QA_{\iv_1}A_{\jv_1}A_{\hbar_1}A_{\hbar_1}A_{\jv_2}A_{\iv_2}Q^{-1}\\
&&=QA_{\iv_1}A_{\jv_1}A_{\hbar_1}A_{\jv_2}A_{\iv_2}Q^{-1}QA_{\iv_2}^{-1}A_{\jv_2}^{-1}A_{\hbar_1}A_{\jv_2}A_{\iv_2}Q^{-1}\\
&&=D_{\iv_1\jv_1\hbar_1\jv_2\iv_2}QA_{\iv_2}^{-1}A_{\jv_2}^{-1}A_{\hbar_1}A_{\jv_2}A_{\iv_2}Q^{-1}.
\end{eqnarray*}
Thus,
\begin{equation}\label{eq:simdiag1}
D_{\hbar_1}=\left(A_{\jv_2}A_{\iv_2}Q^{-1}\right)^{-1}A_{\hbar_1}\left(A_{\jv_2}A_{\iv_2}Q^{-1}\right)\text{ and similarly, }D_{\hbar_1}=\left(QA_{\iv_1}A_{\jv_1}\right)A_{\hbar_1}\left(QA_{\iv_1}A_{\jv_1}\right)^{-1}.
\end{equation}
Moreover, similar argument shows that
\begin{equation}\label{eq:simdiag2}
D_{\hbar_1}=\left(A_{\jv_2'}A_{\iv_2}Q^{-1}\right)^{-1}A_{\hbar_2}\left(A_{\jv_2'}A_{\iv_2}Q^{-1}\right)=\left(QA_{\iv_1}A_{\jv_1'}\right)A_{\hbar_2}\left(QA_{\iv_1}A_{\jv_1'}\right)^{-1}.
\end{equation}
Therefore, the matrices $R_{\hbar_1}^{-1}A_{\jv_2}A_{\iv_2}Q^{-1}$, $R_{\hbar_2}^{-1}A_{\jv_2'}A_{\iv_2}Q^{-1}$, $QA_{\iv_1}A_{\jv_1}R_{\hbar_1}$ and $QA_{\iv_1}A_{\jv_1'}R_{\hbar_2}$ are diagonal matrices. Hence,
\begin{eqnarray*}
D_{\iv_1\jv_1\hbar_1\overline{a}\hbar_2\jv_2'\iv_2}&&=QA_{\iv_1}A_{\jv_1}A_{\hbar_1}A_{\overline{a}}A_{\hbar_2}A_{\jv_2'}A_{\iv_2}Q^{-1}\\
&&=QA_{\iv_1}A_{\jv_1}R_{\hbar_1}R_{\hbar_1}^{-1}A_{\hbar_1}R_{\hbar_1}R_{\hbar_1}^{-1}A_{\overline{a}}R_{\hbar_2}R_{\hbar_2}^{-1}A_{\hbar_2}R_{\hbar_2}^{-1}A_{\jv_2'}A_{\iv_2}Q^{-1}\\
&&=D_1R_{\hbar_1}^{-1}A_{\overline{a}}R_{\hbar_2}D_2,
\end{eqnarray*}
where $D_1$ and $D_2$ are diagonal matrices. Thus $A_{\overline{a}}$ maps the eigendirection $v_{\hbar_2}$ to $v_{\hbar_1}$, which is a contradiction.
\end{proof}

Theorem~\ref{thm:Markov} and Theorem~\ref{thm:triangularmarkov} are immediate corollaries of Theorem~\ref{thm:mainmark}.

\section{Lower bound for the general case with one dimensional base}\label{sec5}

In this section we give the two remaining proofs.

\begin{proof}[Proof of Theorem~\ref{thm:maindiag}] The upper bound follows by the combination of  Lemma~\ref{lem:ubnonmarkov} and equation~\eqref{eq:phipot}.

To show that the lower bound holds, first by the definition \eqref{eq:wholepres} of the pressure $P(f,\varphi^s)$, let $B$ be a Markov subset of $f$ such that $P(f|_B,\varphi^s)$ is sufficiently close to $P(f,\varphi^s)$. Since $k=1$, if  $\{g_i^{-1}\}_{i=1}^M$ satisfies the HESC then every subsystem, formed by composition of functions in $\{g_i^{-1}\}_{i=1}^M$, satisfies it too. Thus, by applying Theorem~\ref{thm:mainmark}\eqref{it:mark1}, we get the assertion.
\end{proof}

\begin{proof}[Proof of Theorem~\ref{thm:maintriang}]
Similarly to the previous case, the upper bound follows by the combination of  Lemma~\ref{lem:ubnonmarkov} and equation~\eqref{eq:phipot}.

To show that the lower bound holds, first by the definition \eqref{eq:wholepres} of the pressure $P(f,\varphi^s)$, let $B$ be a Markov subset of $f$ such that $P(f|_B,\varphi^s)$ is sufficiently close to $P(f,\varphi^s)$.

We may assume that $F|_{B\times\R}$ is essentially non-diagonal. Indeed, let $\iv$ and $\jv$ be the finite length words corresponding to the essentially non-diagonal condition in \eqref{eq:essnondiag}, and let $x_{\iv}$ and $x_{\jv}$ be the two corresponding fixed points. Since $\{x_{\iv}\}$ and $\{x_{\jv}\}$ are trivially Markov subsets of $f$, we can find a Markov subset $B'$ such that $\{x_{\iv},x_{\jv}\}\cup B\subseteq B'$ by Proposition~\ref{prop:joinmarkovs}. Trivially, $P(f|_{B'},\varphi^s)\geq P(f|_B,\varphi^s)$ and $F|_{B'\times\R}$ is essentially non diagonal. Thus, the assertion for the lower bound follows by Theorem~\ref{thm:mainmark}\eqref{it:mark3}.
\end{proof}

\section{Examples: Fractal functions}\label{sec:ex}

Let a data set $\{(x_i, y_i)\in[0,1]\times\R: i=0,1,\ldots,N\}$ be given so that $x_0=0$ and $x_N=1$. Barnsley \cite{Barnsley86} introduced a family of iterated function systems whose attractors $\Lambda$ are graphs of continuous functions $G\colon[0,1]\mapsto\R$, which interpolate the data according to $G(x_i)=y_i$ for $i\in\{0, 1,\ldots,N\}$. This IFS contains only affine transformations with triangular matrices. The dimension theory of the interpolation functions was studied in several papers, see for example Bedford~\cite{Bedford}, Keane, Simon and Solomyak \cite{KSS03} and Ruan, Su and Yao \cite{Ruan}. Here we present a generalised version of fractal interpolation functions $G\colon[0,1]\mapsto\R$ constructed with Markov systems, similar to Deniz and \"Ozdemir \cite{Deniz}.

A particular fractal interpolation function is the Takagi function $G_\lambda\colon[0,1]\mapsto\R$, which is a well known example for a continuous, but nowhere differentiable function, introduced by Takagi, where
$$
G_\lambda(x)=\sum_{n=0}^\infty\lambda^nd_1(2^nx,\mathbb{Z}),
$$
 $1/2<\lambda<1$ and $d_1(x,y)=|x-y|$. It is easy to see that the graph of $G_\lambda$ is the repeller of the dynamics
$$
F(x,y)=\begin{cases}
         \left(2x,\frac{y-x}{\lambda}\right), & \mbox{if }0\leq x<\frac{1}{2}  \\
         \left(2x-1,\frac{y+x-1}{\lambda}\right) & \mbox{if }\frac{1}{2}\leq x\leq 1,
       \end{cases}
$$
which is a Markov system. The dimension of the graph of the Takagi function was studied in several papers, see for example Ledrappier \cite{Ledrappiergraph} and Solomyak \cite{solomyak1998measure}, and the complete answer was given in the recent paper \cite{BHR}. For further properties, see Allaart and Kawamura~\cite{aboutTakagi}.

In the next application, we consider a generalized version of the Takagi function from $[0,1]$ to $\R$, which is not associated to a Markovian system.

A generalisation of the fractal interpolation functions are the fractal interpolation surfaces. For precise definitions, see Feng \cite{Fengsurface} or Bouboulis and Dalla \cite{BouboulisDalla}, Dalla \cite{Dalla}. For $d\geq2$, the defining IFS contains non-linear functions in general. Thus, our method is not suitable for the general case. So, in our last application of our main theorems, we consider an important special case of fractal interpolation surfaces, the graph of multivariable Takagi function.

\subsection{Fractal interpolation functions}\label{sec:interpol}

Let $\{(x_i, y_i)\in[0,1]\times\R: i=0,1,\ldots,N\}$ be a data set such that $0=x_0<x_1<\cdots<x_{N-1}<x_N=1$ and let $I_i=[x_{i-1},x_i]$. For every $i=1,\ldots,N$, let us choose $0\leq k_i<\ell_i\leq N$ such that
$$
\frac{x_{\ell_i}-x_{k_i}}{x_i-x_{i-1}}>1.
$$
We define the base system $f\colon[0,1]\mapsto[0,1]$ such that
\begin{equation}\label{eq:examplebase}
f(x)=\frac{x_{\ell_i}-x_{k_i}}{x_i-x_{i-1}}(x-x_{i-1})+x_{k_i}\text{ if }x\in I_i.
\end{equation}
It is easy to see that $f$ is Markov with respect to the Markov partition $\{I_i\}_{i=1}^N$.

For every  $i=1,\ldots,N$, let us choose real numbers $\lambda_i$ such that
\begin{equation}\label{eq:exparam}
\frac{x_{\ell_i}-x_{k_i}}{x_i-x_{i-1}}>\lambda_i>1.
\end{equation}
Let
\begin{equation}\label{eq:exparam2}
a_i=\frac{y_{\ell_i}-y_{k_i}}{x_i-x_{i-1}}-\lambda_i\frac{y_{i}-y_{i-1}}{x_i-x_{i-1}}\text{ and }d_i=\frac{y_{k_i}x_i-y_{\ell_i}x_{i-1}}{x_i-x_{i-1}}+\lambda_i\frac{y_ix_{i-1}-y_{i-1}x_i}{x_i-x_{i-1}},
\end{equation}
and let us define $g_i(x,y)=a_ix+\lambda_iy+d_i$. Simple calculations show that the system
$$
F(x,y)=(f(x),g_i(x,y))\text{ if }x\in I_i
$$
has a unique repeller $\Lambda$, which is a graph of a continuous function $G$ such that $G(x_i)=y_i$ for $i=0,\ldots,N$. For an example, see Figure~\ref{fig:interpol}.

\begin{figure}
	\centering
	\includegraphics[width=8cm]{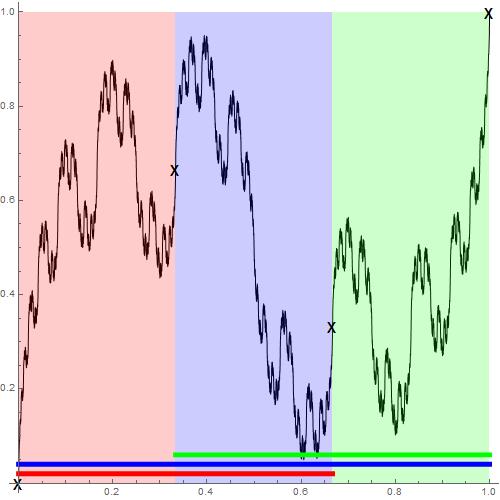}\includegraphics[width=8cm]{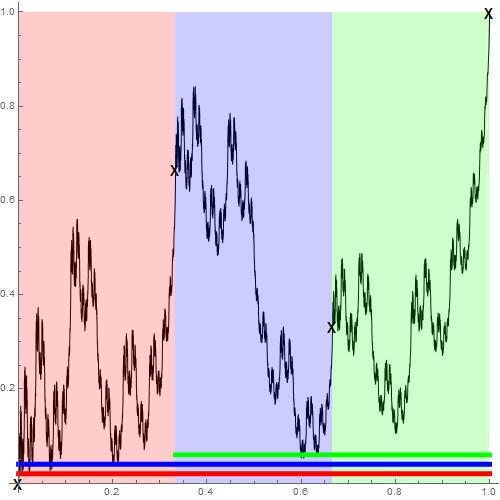}
	\caption{The graph of fractal interpolation functions with data set $\{(0,0),(1/3,2/3),(2/3,1/3),(1,1)\}$, and parameters $\lambda_1=3/2$, $\lambda_2=2$ and $\lambda_3=3/2$. The colors represents which interval is mapped onto which interval by the base system. By Theorem~\ref{thm:interpol}, since $(1+\lambda_2)/(1-\lambda_2)\neq(1-2\lambda_3)/(2-\lambda_3)$ the Hausdorff dimension of the graphs are approximately $1.39024$ and $1.45156$.}\label{fig:interpol}
\end{figure}

Finally, by using Remark~\ref{rem:presspectra}, we define the matrix $A^{(s)}$ so that
\begin{equation}\label{eq:matrixpressure}
A_{i,j}^{(s)}=\begin{cases}
\left(\dfrac{x_i-x_{i-1}}{x_{\ell_i}-x_{k_i}}\right)^{s-1}|\lambda_i| & \text{if }k_i\leq j\leq\ell_i \\
0 & \text{otherwise.}
\end{cases}
\end{equation}
Hence, the root $s_0$ of the pressure defined in \eqref{eq:markovpress} satisfies $\rho(A^{(s_0)})=1$, where $\rho(A)$ denotes the spectral radius of $A$.

\begin{theorem}\label{thm:interpol}
	Let  $\{(x_i, y_i)\in[0,1]\times\R: i=0,1,\ldots,N\}$ and $F$ be such that all the assumptions hold above. If $\lambda_1,\ldots,\lambda_N$ are chosen such that there exist $i,j\in\{1,\ldots,N\}$ such that
	$$
	 \frac{y_{\ell_i}-y_{k_i}-\lambda_i(y_i-y_{i-1})}{x_{\ell_i}-x_{k_i}-\lambda_i(x_i-x_{i-1})}\neq\frac{y_{\ell_j}-y_{k_j}-\lambda_j(y_j-y_{j-1})}{x_{\ell_j}-x_{k_j}-\lambda_j(x_j-x_{j-1})}
	$$
	then
	$$
	\dim_H\Lambda=s_0,
	$$
	where $\rho(A^{(s_0)})=1$ and $A^{(s)}$ is the matrix defined in \eqref{eq:matrixpressure}.
\end{theorem}

The proof of the theorem follows by Theorem~\ref{thm:maintriang}.

\subsection{$\beta$-Takagi function}\label{sec:betataka}

Now, we consider the $\beta$-Takagi functions. That is, let $\beta>1$ and $0<\lambda<1$ so that $\lambda\beta>1$, moreover, let $f_\beta$ be the usual $\beta$-expansion on $[0,1]$, i.e.
$$
f_\beta(x)=\beta x\mod1.
$$
Then let $H_{\beta,\lambda}$ be the function so that
\begin{equation}\label{eq:takabeta}
H_{\beta,\lambda}(x)=\sum_{n=0}^\infty\lambda^nd_1(f_{\beta}^n(x),\mathbb{Z}),
\end{equation}
where $d_1(x,y)=|x-y|$.  Simple calculations show that the graph of $H_{\beta,\lambda}$ is the repeller of the system
\begin{equation}\label{eq:takline}
F(x,y)=\begin{cases}
         \left(\beta x-i+1,\frac{y-x}{\lambda}\right), & \mbox{if }x\in\left[\frac{i-1}{\beta},\frac{i}{\beta}\right)\text{ for }i=1,\ldots,\left\lfloor\frac{\beta}{2}\right\rfloor \\
         \left(\beta x-\left\lfloor\frac{\beta}{2}\right\rfloor,\frac{y-x}{\lambda}\right), & \mbox{if }x\in\left[\frac{\left\lfloor\frac{\beta}{2}\right\rfloor}{\beta},\frac{1}{2}\right) \\
         \left(\beta x-\left\lfloor\frac{\beta}{2}\right\rfloor,\frac{y+x-1}{\lambda}\right), & \mbox{if }x\in\left[\frac{1}{2},\frac{\left\lfloor\frac{\beta}{2}\right\rfloor+1}{\beta}\right) \\
         \left(\beta x-i,\frac{y+x-1}{\lambda}\right), & \mbox{if }x\in\left[\frac{i}{\beta},\frac{i+1}{\beta}\right)\text{ for }i=\left\lceil\frac{\beta}{2}\right\rceil,\ldots,\left\lfloor\beta\right\rfloor-1 \\
         \left(\beta x-\left\lfloor\beta\right\rfloor,\frac{y+x-1}{\lambda}\right), & \mbox{if }x\in\left[\frac{\left\lfloor\beta\right\rfloor}{\beta},1\right].
       \end{cases}
\end{equation}
For an example, see Figure~\ref{fig:takagi2}.

\begin{figure}
	\centering
	\includegraphics[width=9cm]{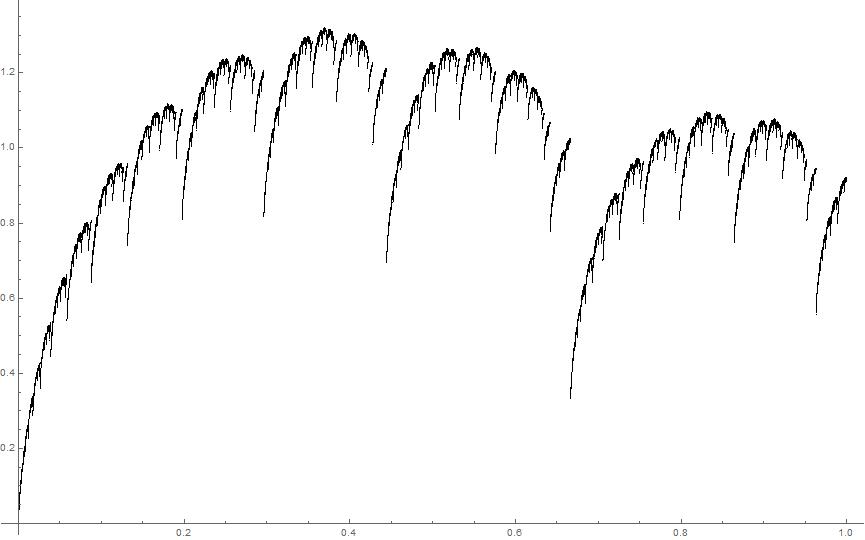}
	\caption{The graph of the $\beta$-Takagi function with parameters $\beta=3/2$ and $\lambda=3/4$. The Hausdorff dimension of the graph is $3-\log2/\log(3/2)\approx1.2905$.}\label{fig:takagi2}
\end{figure}

\begin{theorem}\label{thm:takgeneral}
Let  $\beta>1$ and $0<\lambda<1$ so that $\lambda\beta>1$, and let $H_{\beta,\lambda}$ be as in \eqref{eq:takabeta}. Then
$$
\dim_H\mathrm{graph}(H_{\beta,\lambda})=2+\frac{\log\lambda}{\log\beta},
$$
where $\mathrm{graph}(H_{\beta,\lambda})=\{(x,H_{\beta,\lambda}(x)):x\in[0,1]\}$.
\end{theorem}

\begin{proof}
It is well known that the $\beta$-expansion is topologically transitive for every $\beta>1$. Moreover, let $X\subset\{0,1,\ldots,\lceil\beta\rceil\}^\N$ defined in \eqref{eq:symbshift}, and $X_n$ be all the $n$th level cylinders intersecting $X$. Then by R\'enyi \cite[equation (4.9) and (4.10)]{Renyi},
$$
h_{top}(f_\beta)=\lim_{n\to\infty}\frac{\log X_n}{n}=\log\beta.
$$

Now, we show that $F$ is essentially non-diagonal. Let us choose $n\geq1$ such that
\begin{equation}\label{eq:betatak}
\frac{\lceil\beta\rceil-1}{\beta^{n+1}-1}<\frac{1}{\beta^n}.
\end{equation}
For short, let $x_n=\frac{\lceil\beta\rceil-1}{\beta^{n+1}-1}$ and let $y_n=H_{\beta,\lambda}(x_n)$. Inequality \eqref{eq:betatak} implies that $f^{n+1}\left(x_n\right)=x_n$. Let $$k=\min\{\ell\geq0:f^\ell(x_n)\geq1/2\}.$$
Since $\frac{\lfloor\beta\rfloor}{\beta}\geq\frac{1}{2}$ and $f^{n}(x_n)\in\left[\frac{\lfloor\beta\rfloor}{\beta},1\right]$, we get that $k\leq n$. Moreover,
\begin{eqnarray*}
D_{(x_n,y_n)}F^{n+1}&=\overbrace{\left(\begin{matrix}
\beta & 0 \\
\frac{-1}{\lambda} & \frac{1}{\lambda}\end{matrix}\right)\cdots\left(\begin{matrix}
\beta & 0 \\
\frac{-1}{\lambda} & \frac{1}{\lambda}\end{matrix}\right)}^{k\text{ times}}\cdot\overbrace{\left(\begin{matrix}
\beta & 0 \\
\frac{1}{\lambda} & \frac{1}{\lambda}\end{matrix}\right)\cdots\left(\begin{matrix}
\beta & 0 \\
\frac{1}{\lambda} & \frac{1}{\lambda}\end{matrix}\right)}^{n-k+1\text{ times}}\\
&=\left(\begin{matrix}
\beta^{n+1} & 0 \\
-\sum_{\ell=1}^k\frac{\beta^{n+1-\ell}}{\lambda^\ell}+\sum_{\ell=k+1}^{n+1}\frac{\beta^{n+1-\ell}}{\lambda^\ell} & \frac{1}{\lambda^{n+1}}\end{matrix}\right).
\end{eqnarray*}
On the other hand, $0$ is a fixed point of $f$, and thus,
$$
D_{(0,0)}F^{n+1}=\left(\begin{matrix}
\beta^{n+1} & 0 \\
-\sum_{\ell=1}^{n+1}\frac{\beta^{n+1-\ell}}{\lambda^\ell} & \frac{1}{\lambda^{n+1}}\end{matrix}\right).
$$
Since $k\leq n$, the eigendirections of $D_{(x_n,y_n)}F^{n+1}$ and $D_{(0,0)}F^{n+1}$ have different eigendirections. On the other hand, $f_{\beta}|_{[0,1/\beta]}$ has full stripe and therefore the path $0^{n+1}00^{k}\lceil\beta\rceil^{n-k+1}$ is admissible and $D_{(0,0)}F$ does not map the eigendirection of $D_{(0,0)}F^{n+1}$ to the eigendirection of $D_{(x_n,y_n)}F^{n+1}$. Thus, \eqref{eq:essnondiag} holds.

By applying Theorem~\ref{thm:maintriang},
$$
\dim_H\mathrm{graph}(H_{\beta,\lambda})=s_0,
$$
where $s_0$ is the unique root of the pressure $P(f_\beta,\varphi^s)=\sup_{B\in\MM(A)}P(f_\beta|_B,\varphi^s)$. By Lemma~\ref{lem:ubandadditive}, $s_0\leq2+\frac{\log\lambda}{\log\beta}$. Thus, it is enough to find a sequence of Markov subsets $B_m$, for which $s_m\to2+\frac{\log\lambda}{\log\beta}$, where $s_m$ is the unique root of the pressure $s\mapsto P(f|_{B_m},\varphi^{s})$.

Let us denote the set of continuity intervals of $f_{\beta}$ by $\mathcal{I}$, that is, $$
\mathcal{I}=\left\{\left[\frac{i}{\beta},\frac{i+1}{\beta}\right]:i=0,\ldots,\lfloor\beta\rfloor-1\right\}\cup\left\{\left[\frac{\lfloor\beta\rfloor}{\beta},1\right]\right\},
$$
and the $n$th refinement of $\II$ by $\II_n=\bigvee_{i=0}^{n-1}f_{\beta}^{-i}(\II)$.

\begin{claim}
For every $\varepsilon>0$ there exists $m\geq1$, a set $B_m\subset[0,1]$ and $\DD_m\subseteq\II_m$ such that
\begin{enumerate}
  \item $f_\beta(B_m)=B_m$ and $f_\beta|_{B_m}$ topologically transitive,
  \item $B_m$ is a Markov subset with Markov partition $\DD_m$,
  \item $h_{top}(f_\beta|_{B_m})>h_{top}(f_\beta)-\varepsilon$.
\end{enumerate}
\end{claim}

The claim follows from Hofbauer, Raith and Simon \cite[Proposition~1(a),(b),(c) and Lemma~2]{hofraithsim}.

Let $A^{(s)}$ be a $\#\DD_m\times\#\DD_m$ matrix such that
$$
A^{(s)}_{I,J}=\begin{cases}\lambda\beta^{-(s-1)} & \text{if }J\cap B_m\subseteq f_{\beta}(I\cap B_m)\text{ for }I,J\in\DD_m \\
                            0 & \text{otherwise}.\end{cases}
$$
By Remark~\ref{rem:presspectra}, $\rho(A^{(s_m)})=1$, where $s_m$ is the root of $s\mapsto P(f_\beta|_{B_m},\varphi^s)$. Since $f_\beta|_{B_m}$ is topologically transitive, there exists $K\geq1$ such that every element of $\left(A^{(s_m)}\right)^K$ is strictly positive and by the Perron-Frobenius Theorem, $\lim_{k\to\infty}\left(A^{(s_m)}\right)^k=\underline{u}\underline{v}^T$, where $\underline{u}$ and $\underline{v}$ are the right- and left-eigenvectors of $A^{(s_m)}$ with eigenvalue $1$ so that $\underline{v}^T\underline{u}=1$.

For, $I,J\in\DD_m$, let
$$
I\stackrel{n}{\rightarrow}J=\left\{(I_1,\ldots,I_{n}):I_j\in\DD_m,\ I_1=I, I_n=J,\ f_\beta(I_j\cap B_m)\supseteq I_{j+1}\cap B_m\text{ for }1\leq j\leq n-1\right\}.
$$
Thus,
$$
h_{top}(f_\beta|_{B_m})=\lim_{n\to\infty}\frac{\log\#\bigcup_{I,J\in\DD_m}I\stackrel{n}{\rightarrow}J}{n}.
$$
But for every $k\geq 1$, and $I,J\in\DD_m$,
$$
\left(\left(A^{(s_m)}\right)^k\right)_{I,J}=\left(\lambda\beta^{-(s_m-1)}\right)^k\cdot\#(I\stackrel{n}{\rightarrow}J).
$$
Hence,
$$
\log\beta-\varepsilon<h_{top}(f_\beta|_{B_m})=\lim_{k\to\infty}\frac{\log\frac{\boldsymbol{1}^T\left(A^{(s_m)}\right)^k\boldsymbol{1}}{\left(\lambda\beta^{-(s_m-1)}\right)^k}}{k}=-\log\left(\lambda\beta^{-(s_m-1)}\right),
$$
which implies that $s_m>2+\frac{\log\lambda}{\log\beta}-\varepsilon.$ Since $\varepsilon>0$ was arbitrary, the statement follows.

\end{proof}

\subsection{Multivariable Takagi function}\label{sec:multitaka}

Let $d\geq1$ and let $D=\mathrm{diag}(\beta_1,\ldots,\beta_d)$ be a diagonal matrix such that $\beta_i\geq2$ integers for $i=1,\ldots,d$. For $(\max_i\beta_i)^{-1}<\lambda<1$, let $G_{\lambda,D}\colon[0,1]^d\mapsto\R$ be such that
$$
G_{\lambda,D}(\underline{x})=\sum_{n=0}^\infty\lambda^nd_1(D^n\underline{x},\mathbb{Z}^d),
$$
where $d_1(\underline{x},\underline{y})=\sum_{i=1}^d|x_i-y_i|$. Denote $\mathrm{graph}(G_{\lambda,D})$ the graph of $G_{\lambda,D}$, i.e.
$$
\mathrm{graph}(G_{\lambda,D})=\left\{(\xv,y)\in[0,1]^d\times\R:G_{\lambda,D}(\xv)=y\right\}.
$$
See Figure~\ref{fig:takagi}, for an example of the multivariable Takagi function.

Similarly to \eqref{eq:takline}, $\mathrm{graph}(G_{\lambda,D})$ is the unique invariant repeller of the map\linebreak $F\colon[0,1]^d\times\R\mapsto[0,1]^d\times\R$, where
\begin{equation}\label{eq:gentak}
F(x_1,\ldots,x_d,y)=\left(\beta_1x_1\mod1,\ldots,\beta_dx_d\mod1,\frac{y-d_1(\xv,\mathbb{Z}^d)}{\lambda}\right).
\end{equation}
Moreover,
\begin{equation}\label{eq:struktak}
G_{\lambda,D}(x_1,\ldots,x_d)=H_{\beta_1,\lambda}(x_1)+\cdots+H_{\beta_d,\lambda}(x_d).
\end{equation}
The main statement of this section is the following.

\begin{figure}
	\centering
	\includegraphics[width=\textwidth]{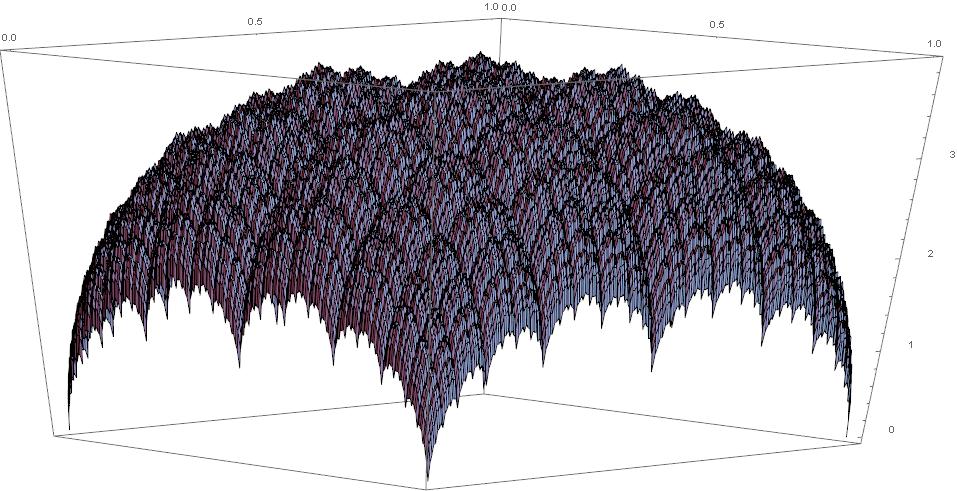}
	\caption{The graph of the multivariable Takagi function with parameters $d=2$, $\beta_1=\beta_2=2$ and $\lambda=2/3$.}\label{fig:takagi}
\end{figure}

\begin{theorem}\label{thm:takagi}
  Let $d\geq1$ and let $D=\mathrm{diag}(\beta_1,\ldots,\beta_d)$ be a diagonal matrix such that $\beta_i\geq2$ integers for $i=1,\ldots,d$. For $(\max_i\beta_i)^{-1}<\lambda<1$
  $$
  \dim_H\mathrm{graph}(G_{\lambda,D})=d+1+\frac{\log\lambda}{\log\max_i\beta_i}.
  $$
\end{theorem}

However, the proof of Theorem~\ref{thm:takagi} is quite ad-hoc and using deeply the special structure \eqref{eq:struktak} of the function $G_{\lambda,D}$. Thus, to illustrate the application of Theorem~\ref{thm:triangularmarkov}, we present here a weaker result also, which might be instructive for further applications.

\begin{proposition}\label{thm:takagiminor}
  Let $d\geq2$ and $D'=\mathrm{diag}(2,\ldots,2)$. Then there exists a set $E\subset(\frac{1}{2},1)$ such that $\dim_PE=0$ and
  $$
  \dim_H\mathrm{graph}(G_{\lambda,D'})=d+1+\frac{\log\lambda}{\log2}\text{ for every }\lambda\in(1/2,1)\setminus E.
  $$
  Moreover, if $\frac{1}{\sqrt{5}-1}<\lambda<1$ then $\dim_H\mathrm{graph}(G_{\lambda,D'})=d+1+\frac{\log\lambda}{\log2}$.
\end{proposition}

By \eqref{eq:gentak}, observe that $\mathrm{graph}(G_{\lambda,D'})$ is the attractor of the IFS
\begin{equation}\label{eq:ifstakagi}
\Phi=\left\{F_{\iv}(\xv)=\left(\begin{matrix}
  1/2 & 0 & \cdots & 0  \\
  0 & \ddots & 0 & 0 \\
  0 & \cdots & 1/2 &0 \\
  (-1)^{i_1}/2 & \cdots & (-1)^{i_n}/2 & \lambda
\end{matrix}\right)\xv+\left(\begin{matrix}
  i_1/2 \\
  \vdots \\
  i_n/2 \\
  \sum_{k=1}^di_k
\end{matrix}\right)\right\}_{\iv\in\{0,1\}^d}.
\end{equation}
By using the definition of the pressure \eqref{eq:compare} and \eqref{eq:markovpress}, we get for the root of the pressure that
$$
s_0=d+1+\frac{\log\lambda}{\log 2}.
$$
Moreover, the equilibirum measure $\mu$ (defined in Lemma~\ref{lem:vari}) is the uniform Bernoulli measure on the symbolic space $\Sigma=\left(\{0,1\}^d\right)^\N$.

According to Section~\ref{sec:tools} and to \eqref{eq:furstsys}, the Furstenberg-Kifer measure $\mu_F$ on $\R^d$ is associated to the uniform Bernoulli measure with the IFS
$$
\Phi_F=\left\{h_{\iv}(\xv)=\frac{1}{2\lambda}\xv+\left(\begin{matrix}
  (-1)^{i_1}/2 \\
  \vdots \\
  (-1)^{i_n}/2
\end{matrix}\right)\right\}_{\iv\in\{0,1\}^d}.
$$
It is easy to see that $\Phi_F$ does not satisfy Definition~\ref{def:hochman}\eqref{item:HESC2} for $d\geq2$. So, we cannot apply Theorem~\ref{thm:triangularmarkov} and we need a more sophisticated analysis.

By Theorem~\ref{thm:rap}, it is enough to show that $s_0+\dim_H\mu_F>d+1$. Since $s_0\geq d$, in order to prove Theorem~\ref{thm:takagi}, it is enough to show the following lemma.

\begin{lemma}
	Let $\mu_F$ be the Furstenberg-Kifer measure defined above. That is,
	$$
	\mu_F=\sum_{\iv\in\{0,1\}^d}\frac{1}{2^d}\mu_F\circ h_{\iv}^{-1}.
	$$
	Then there exists a set $E\subset(1/2,1)$ such that $\dim_PE=0$ and $\dim_H\mu_F\geq1$ for every $\lambda\in(1/2,1)\setminus E$. Moreover, $\dim_H\mu_F>1$ for $\lambda\in(1/(\sqrt{5}-1),1)$.
\end{lemma}

\begin{proof}
	Let us prove the first assertion. Observe that the orthogonal projection of $\mu_F$ to the first coordinate axis is the self-similar measure on the real line with respect to the IFS $\{x\mapsto \frac{1}{2\lambda}x+\frac{1}{2},x\mapsto \frac{1}{2\lambda}x-\frac{1}{2}\}$ and probability vector $\{\frac{1}{2},\frac{1}{2}\}$. Thus, by applying \cite[Theorem~1.9]{hochman2012self}, we get $\dim_H\mu_F\geq\dim_H\proj_*\mu_F=1$ outside of a set with $0$ dimension.
	
	In order to show the second assertion, let us introduce a few notations. Let $\Lambda=[0,1]^d$ and let us denote the $n$th cylinder sets by $\Lambda_\ii=h_{\ii}(\Lambda)$ for $\ii=(\iv_1,\ldots,\iv_n)$. Observe that if $\frac{1}{2\lambda}+\left(\frac{1}{2\lambda}\right)^2<1$ then for every $x\in[0,1]^d$
	$$
	\#\{\Lambda_{\iv_1\iv_2}:x\in\Lambda_{\iv_1\iv_2}\}\leq 2^d.
	$$
	Thus, by choosing $\kappa=\min\{d(\Lambda_{\iv_1\iv_2},\Lambda_{\jv_1,\jv_2}):\Lambda_{\iv_1\iv_2}\cap\Lambda_{\jv_1,\jv_2}=\emptyset\}/2>0$, for every $x\in[0,1]^d$ and $n\geq1$
	$$
	\mu(B_{\kappa\left(\frac{1}{2\lambda}\right)^{2n}}(x))\leq\frac{\#\left\{\ii\in\left(\{0,1\}^d\right)^{2n}:\Lambda_{\ii}\cap B_{\kappa\left(\frac{1}{2\lambda}\right)^{2n}}(x)\neq\emptyset\right\}}{2^{2dn}}\leq\frac{1}{2^{dn}},
	$$
	and hence,
	$$
	 \dim_H\mu_F\geq\inf_x\liminf_{n\to\infty}\frac{\log\mu_F(B_{\kappa\left(\frac{1}{2\lambda}\right)^{2n}}(x))}{\log\kappa\left(\frac{1}{2\lambda}\right)^{2n}}=\frac{d\log2}{2\log2+2\log\lambda}>1
	$$
	for $\lambda<1$ and $d\geq2$.
\end{proof}

Finally, we turn to the proof of Theorem~\ref{thm:takagi}, which follows by the next two lemmas.

\begin{lemma}\label{lem:totakag}
  Let $d\geq2$ and let $g_i\colon[0,1]\mapsto\R$ be functions for $i=1,\ldots,d$ such that
  $$
  \dim_H\{(x,y):x\in[0,1]\text{ and }g_i(x)=y\}\geq2-\delta_i.
  $$
  Then
  $$
  \dim_H\left\{(x_1,\ldots,x_d,y):(x_1,\ldots,x_d)\in[0,1]^d\text{ and }\sum_{i=1}^dg_i(x_i)=y\right\}\geq d+1-\min_i\delta_i.
  $$
\end{lemma}

\begin{proof}
For simplicity, let $G(x_1,\ldots,x_d)=g_1(x_1)+\cdots+g_d(x_d)$. Let $i\in\{1,\ldots,d\}$ be arbitrary. Then for every fixed $\xv'=(x_1,\ldots,x_{i-1},x_{i+1},\ldots,x_d)$
$$
\Gamma_{\xv'}^i=\{(x_i,y)\in[0,1]\times\R: G(x_1,\ldots,x_d)=y\}=\{(x_i,y)\in[0,1]\times\R:g_i(x_i)=y-\sum_{\substack{k=1 \\ k\neq i}}^dg_i(x_i)\}
$$
is a translation of $\mathrm{graph}(g_i)=\{(x,y):x\in[0,1]\text{ and }g_i(x)=y\}$. Hence, for every $i\in\{1,\ldots,d\}$ and $\xv'\in[0,1]^{d-1}$, $\dim_H\Gamma_{\xv'}^i=2-\delta_i$.

By \cite[Theorem~7.7]{mattila1999geometry}, there exists $C>0$ such that for every $\varepsilon>0$
$$
\infty=\int\mathcal{H}^{2-\delta_i-\varepsilon}(\Gamma_{\xv'}^i)d\mathcal{L}_{d-1}(\xv')\leq C\mathcal{H}^{d+1-\delta_i-\varepsilon}(\mathrm{graph}(G)).
$$
Thus, $\dim_H\mathrm{graph}(G)\geq d+1-\min_i\delta_i$.
\end{proof}

The following lemma is folklore but for completeness, we give here the complete proof.

\begin{lemma}\label{lem:totakag2}
	Let $\beta_1,\ldots,\beta_d\geq2$ integers and let $1>\lambda>(\max_i\beta_i)^{-1}$. Then the function $G_{\lambda,D}$ is H\"older continous with exponent $-\frac{\log\lambda}{\log\max\beta_i}$ and
	$$
	\overline{\dim}_B\mathrm{graph}(G_{\lambda,D})\leq d+1+\frac{\log\lambda}{\log\max_i\beta_i}.
	$$
\end{lemma}

\begin{proof}
	First, we show that for every $\beta\geq2$ integer and $\beta^{-1}<\lambda<1$, the function $H_{\beta,\lambda}$ is H\"older continuous with exponent $\frac{\log\lambda}{\log\beta}$. It is easy to see that
	$$
	 d(\beta^nx,\mathbb{Z})=\min_{k\in\mathbb{Z}}|\beta^nx-k|\leq\beta^n|x-y|+\min_{k\in\mathbb{Z}}|\beta^ny-k|=\beta^n|x-y|+d(\beta^ny,\mathbb{Z}).
	$$
 	Let $x,y\in[0,1]$ and $k\geq0$ such that $\beta^{-k-1}<|x-y|\leq\beta^{-k}$. Hence,
 \begin{eqnarray}
|H_{\beta,\lambda}(x)-H_{\beta,\lambda}(y)|&\leq&\sum_{n=0}^k\lambda^n|d(\beta^nx,\mathbb{Z})-d(\beta^ny,\mathbb{Z})|+2\frac{\lambda^{k+1}}{1-\lambda}\\
&\leq&\sum_{n=0}^k\lambda^n\beta^n|x-y|+2\frac{\lambda^{k+1}}{1-\lambda}\\
&\leq&\frac{\lambda^{k+1}\beta^{k+1}}{\lambda\beta-1}|x-y|+2\frac{\lambda^{k+1}}{1-\lambda}\\
&\leq&\left(\frac{\beta}{\lambda\beta-1}+\frac{2}{1-\lambda}\right)|x-y|^{-\log\lambda/\log\beta}.
\end{eqnarray}
Similarly, one can show if $\lambda\leq\beta^{-1}$ then $H_{\beta,\lambda}$ is $(1-\varepsilon)$-H\"older for every $\varepsilon>0$. Thus, by choosing $\varepsilon>0$ sufficiently small,
\[
\begin{split}
|G_{\lambda,D}(\xv)-G_{\lambda,D}(\yv)|&\leq\sum_{i=1}^d|H_{\beta_i,\lambda}(x_i)-H_{\beta_i,\lambda}(y_i)|\\
&\leq \sum_{i=1}^dC_i|x_i-y_i|^{-\log\lambda/\log\max_i\beta_i}\leq C'\|\xv-\yv\|_1^{-\log\lambda/\log\max_i\beta_i}.
\end{split}
\]

To show the second claim of the lemma, let us divide $[0,1]^d$ into cubes $\{U_i\}_{i=1}^{2^{nd}}$ with sidelength $2^{-n}$. For $A\subset\R^{d+1}$, denote $N_n(A)$ the minimal number of cubes with sidelength $2^{-n}$ needed to cover $A$. By the H\"older-continuity of $G_{\lambda,D}$, $G_{\lambda,D}(U_i)$ can be covered by at most $C'2^{-n(-\log\lambda/\log\max_i\beta_i)-n}+1$ intervals with length $2^{-n}$. Thus,
$$
N_n(\mathrm{graph}(G_{\lambda,D}))\leq (C'2^{-n(-\log\lambda/\log\max_i\beta_i)-n}+1)2^{-nd},
$$
which completes the proof.
\end{proof}

\begin{proof}[Proof of Theorem~\ref{thm:takagi}]
The bound $\dim_H\mathrm{graph}(G_{\lambda,D})\leq d+1+\frac{\log\lambda}{\log\max_i\beta_i}$ follows by Lemma~\ref{lem:totakag2}. The lower bound follows by the combination of Lemma~\ref{lem:totakag} and Theorem~\ref{thm:takgeneral}.
\end{proof}

\bibliographystyle{plain}
\bibliography{barnsley_08_08}

\end{document}